\documentclass{amsart}

\usepackage{amssymb}
\usepackage{amsthm}
\usepackage{indentfirst}
\usepackage{amsmath}
\usepackage{amscd}
\usepackage[all]{xy}
\usepackage{xypic}
\usepackage{mathtools}
\usepackage[colorlinks, citecolor=blue]{hyperref}

\numberwithin{equation}{section}

\newtheorem{Theorem}{Theorem}[section]
\newtheorem{Definition}[Theorem]{Definition}
\newtheorem{Proposition}[Theorem]{Proposition}
\newtheorem{Lemma}[Theorem]{Lemma}
\newtheorem{Notation}[Theorem]{Notation}

\newtheorem{Assumption-Notation}[Theorem]{Assumption-Notation}

\newtheorem{Remark}[Theorem]{Remark}
\newtheorem{Corollary}[Theorem]{Corollary}
\newtheorem{Claim}[Theorem]{Claim}

\newtheorem{Example}[Theorem]{Example}
\newtheorem{Conjecture}[Theorem]{Conjecture}

\newtheorem*{Acknowledgments}{Acknowledgments}

\begin{document}

\title[Subadditivity of Kodaira dimensions]{Subadditivity of Kodaira dimensions for fibrations of three-folds in positive characteristics}

\address{Lei Zhang\\Key Laboratory of Wu Wen-Tsun Mathematics, Chinese Academy
of Sciences\\ School of Mathematical Science\\University of Science and Technology of China\\Hefei 230026, P.R.China.}
\email{zhlei18@ustc.edu.cn, zhleimath@163.com}
\author{Lei Zhang}
\maketitle

\begin{abstract}
In this paper, we will prove subadditivity of Kodaira dimensions for a fibration with possibly singular geometric generic fiber, under certain nefness and relative semi-ampleness conditions. As an application, for a fibration $f: X \to Y$ of a smooth projective threefold over an algebraically closed field of characteristic $p>5$, under the assumption that $Y$ is of general type and non-uniruled, we prove subadditivity of Kodaira dimensions when general fibers are smooth or when $K_{X/Y}$ is relatively big over $Y$.

\emph{Keywords}: Kodaira dimension; positive characteristic; weak positivity; minimal model.\\
\emph{MSC}: 14E05; 14E30.
\end{abstract}

\section{Introduction}
Let $X$ be a projective variety over a field $k$, $D$ a $\mathbb{Q}$-Cartier divisor on $X$. The \emph{$D$-dimension} $\kappa(X,D)$ is defined as
\[\kappa(X,D) =\left\{
\begin{array}{llr}
-\infty, \text{ ~~if for every integer } m >0, |mD| = \emptyset;\\
\max \{\dim_k \Phi_{|mD|}(X)| m \in \mathbb{Z}~\text{and}~m>0 \}, \text{ otherwise.}
\end{array}\right.
\]
If $X$ has a smooth projective birational model $\tilde{X}$, the \emph{Kodaira dimension} $\kappa(X)$ of $X$ is defined as $\kappa(\tilde{X}, K_{\tilde{X}})$
where $K_{\tilde{X}}$ denotes the canonical divisor. Kodaira dimension is one of the most important birational invariant in the classification theory.

Let $f: X \rightarrow Y$ be a morphism between two schemes. For $y \in Y$, let $X_y$ denote the fiber of $f$ over $y$; and for a divisor $D$ (resp. a sheaf $\mathcal{F}$) on $X$, let $D_y$ (resp. $\mathcal{F}_y$) denote the restriction of $D$ (resp. $\mathcal{F}$) on the fiber $X_y$. Throughout this paper, since $Y$ frequently appears as an integral scheme, we use the special notation $\eta$ and $\bar{\eta}$ for the generic and geometric generic point of $Y$ respectively. We say $f$ is a \emph{fibration} if $f$ is a projective morphism such that $f_*\mathcal{O}_X = \mathcal{O}_Y$.

For the fibrations over $\mathbb{C}$, the following problem is of great importance in birational geometry.
\begin{Conjecture}[Iitaka conjecture] Let $f:X\rightarrow Y$ be a fibration between two smooth projective varieties over $\mathbb{C}$, with $\dim X = n$ and $\dim Y = m$. Then
  $$C_{n,m}: \kappa(X)\geq \kappa(Y) + \kappa(X_{\bar{\eta}}).$$
\end{Conjecture}
This conjecture has been studied by Kawamata (\cite{Ka0}, \cite{Ka1}, \cite{Ka2}),  Koll\'ar (\cite{Ko}), Viehweg (\cite{Vie1}, \cite{Vie2}, \cite{Vie2}), Birkar (\cite{Bir09}), Chen and Hacon (\cite{CH}), Cao and P\v{a}un (\cite{CP0}) etc..

In positive characteristics, analogously it is conjectured that
\begin{Conjecture}[Weak Subadditivity]\label{wc} Let $f:X\rightarrow Y$ be a fibration between smooth projective varieties over an algebraically closed field $k$ of positive characteristic, with $\dim X = n$ and $\dim Y = m$. Assume that the geometric generic fibre $X_{\bar{\eta}}$ is integral and has a smooth projective birational model $\tilde{X}_{\bar{\eta}}$. Then
  $$WC_{n,m}: \kappa(X)\geq \kappa(Y) + \kappa(\tilde{X}_{\bar{\eta}}).$$
\end{Conjecture}
\begin{Remark}
The condition that $X_{\bar{\eta}}$ is integral is equivalent to that $X_{\bar{\eta}}$ is reduced, and also is equivalent to that $f$ is separable by \cite[Sec. 3.2.2]{Liu}. If $\dim Y = 1$ then the fibration $f$ is  separable by \cite[Lemma 7.2]{ba01}, thus $X_{\bar{\eta}}$ is integral.

The reason why we assume the existence of smooth birational models is to guarantee that $WC_{n,m}$ makes sense, because the geometric generic fibre $X_{\bar{\eta}}$ is not necessarily smooth (which is true over $\mathbb{C}$). In positive characteristics, smooth resolution of singularities has been proved only in dimension $\leq 3$ (\cite{CP1} and \cite{CP2}). Here we mention that Luo proposed a new definition \cite[Def. 5.1]{Luo87} of the Kodaira dimension of a variety $X$ via its function field $K(X)$, without involving smooth resolutions. This definition coincides with the traditional one when $X$ is a smooth projective variety. For more discussions please refer to \cite[Appendix B]{Pa2}.
\end{Remark}

Notice that if both $X$ and $Y$ are smooth, then the dualizing sheaf of $X_{\bar{\eta}}$ is invertible, thus the canonical divisor $K_{X_{\bar{\eta}}}$ is Cartier. It is reasonable to ask whether the following is true.
\begin{Conjecture}\label{sc} Let $f:X\rightarrow Y$ be a fibration between smooth projective varieties over an algebraically closed field $k$ of positive characteristic,
with $\dim X = n$ and $\dim Y = m$. Then
  $$C_{n,m}: \kappa(X)\geq \kappa(Y) + \kappa(X_{\bar{\eta}}, K_{X_{\bar{\eta}}}).$$
\end{Conjecture}
It is known that $C_{n,m}$ implies $WC_{n,m}$ by \cite[Corollary 2.5]{CZ}, and we call the inequality $WC_{n,m}$ weak subadditivity. Up to some Frobenius base changes and a smooth resolution, to prove $WC_{n,m}$ is equivalent to prove $C_{n,m}$ for another fibration with smooth geometric generic fiber (\cite[proof of Corollary 1.3]{BCZ}). It is easier to treat a fibration with smooth geometric generic fiber, because then one can take advantage of moduli theory and positivity results proved
recently by Patakfalvi \cite{Pa} and Ejiri \cite{Ej}. Using these technical results, the following have been proved:
\begin{itemize}
\item[(i)]{$WC_{n, n-1}$ by Chen and Zhang (\cite{CZ});}
\item[(ii)]{$WC_{3,1}$ by Birkar, Chen and Zhang over $\bar{\mathbb{F}}_p, p >5$ (\cite{BCZ});}
\item[(iii)]{$WC_{3,1}$ under the situation that $\tilde{X}_{\bar{\eta}}$ is of general type and $\mathrm{char}~k >5$ by Ejiri (\cite{Ej}).}
\end{itemize}
When the geometric generic fiber is singular, the only known result is $C_{2,1}$, which follows from Bombieri-Mumford's classification of surfaces (cf. \cite{CZ}). And recently Patakfalvi proves $C_{n, m}$ under the situation that $f$ is separable, $\dim_{k(\bar{\eta})} S^0(X_{\bar{\eta}}, K_{X_{\bar{\eta}}})>0$ and $K_Y$ is big (\cite{Pa2}).

This paper aims to treat the fibrations with possibly singular geometric generic fibers. Our main result is the following theorem.
\begin{Theorem}\label{mthk}
Let $f:X\rightarrow Y$ be a separable fibration between two normal projective varieties over an algebraically closed field $k$ with $\mathrm{char}~k = p>0$. Assume either that $Y$ is smooth or that $f$ is flat.
Let $D$ be a Cartier divisor on $X$.

If there exist an effective $\mathbb{Q}$-Weil divisor $\Delta$ on $X$ and a big $\mathbb{Q}$-Cartier divisor $A$ on $Y$ such that

(1) \small{$K_X+ \Delta$} is $\mathbb{Q}$-Cartier and the Cartier index \small{$\mathrm{ind}((K_X+ \Delta)_{\eta})$} is indivisible by $p$;

(2) $D - K_{X/Y} - \Delta - f^*A$ is nef and $f$-semi-ample;

(3) $\dim_{k(\bar{\eta})} S^0_{\Delta_{\bar{\eta}}}(X_{\bar{\eta}}, D_{\bar{\eta}}) > 0$,

then
$$\kappa(X, D) \geq \dim Y + \kappa(X_{\bar{\eta}}, D_{\bar{\eta}}).$$

In particular, if $D$ is nef and $f$-big, and conditions (1) and

(2') $D - K_{X/Y} - \Delta - f^*A$ is nef\\
hold, then $D$ is big.
\end{Theorem}

\begin{Remark}
If setting $\Delta = 0, D = K_X$ and $A = K_Y$, by Theorem \ref{mthk} we get the main result of \cite{Pa2} mentioned before. The condition (3) above holds if $D_{\bar{\eta}}$ is sufficiently big (Proposition \ref{F-non-vanishing}). In the application to the study of Kodaira dimension, if $K_X$ is not relatively big, a strategy is to consider the relative Iitaka fibration, but then some kind of canonical bundle formula is needed (Section \ref{can-bdl-formula}).
\end{Remark}

\begin{Remark}
For a separable fibration $f:X\rightarrow Y$, there always exists a projective birational morphism $Y'\rightarrow Y$ such that the main component $X'$ of $X\times_Y Y'$ is flat over $Y'$ $($\cite[Lemma 3.4]{AO}$)$. So it is convenient to pass to a flat fibration. The advantage of flat fibrations lies in that the relative canonical sheaves are compatible with base changes $($cf. Proposition \ref{compds}$)$.
\end{Remark}

As an easy consequence we get
\begin{Corollary}\label{app-to-3dim}
Let $f:X\rightarrow Y$ be a separable fibration from a normal projective 3-fold to a normal surface or a curve over an algebraically closed field $k$ with $\mathrm{char}~k = p>0$. Let $\Delta$ be an effective divisor on $X$ such that $K_X + \Delta$ is $\mathbb{Q}$-Cartier, nef and $f$-big.
Let $\mu: Z \to Y$ be a smooth resolution.
If $K_Z$ is big, then $K_X + \Delta$ is big.
\end{Corollary}

Combining the recent results of minimal model theory in dimension 3 (cf. \cite{HX}, \cite{Bir13}), we can prove
\begin{Corollary}\label{app-to-3dim-special}
Let $(X, \Delta)$ be a projective klt pair of dimension 3, and let $f: X \rightarrow Y$ be a separable fibration to a smooth projective curve or a surface, over an algebraically closed field $k$ with $\mathrm{char}~k = p >5$. Assume that $K_Y$ is big and $Y$ is non-uniruled. Then
$$\kappa(X, K_X + \Delta) \geq \kappa(Y) + \kappa(X_{\bar{\eta}}, K_{X_{\bar{\eta}}} + \Delta_{\bar{\eta}})$$
if one of the following holds

(1) $K_{X/Y} + \Delta$ is $f$-big;

(2) $\Delta = 0$ and the geometric generic fiber $X_{\bar{\eta}}$ is smooth.
\end{Corollary}

\begin{Remark}
(1) In Corollary \ref{app-to-3dim}, as $K_X +\Delta$ is assumed to be nef, the pair $(X,\Delta)$ is not necessarily assumed to be klt. This result holds in arbitrary dimensions if granted smooth resolution of singularities.

(2) When $Y$ is of general type and non-uniruled, Corollary \ref{app-to-3dim-special} implies $WC_{3,n}$ by \cite[Corollary 2.5]{CZ}, and $C_{3,n}$ if in addition $K_{X/Y}$ is $f$-big. Shortly after this paper was written, Ejiri and the author prove $WC_{3,n}$ completely in \cite{EZ16}, they treat the case $g(Y) =1$ by a very clever use of trace maps and a deep result of vector bundles on curves.

(3) Varieties of maximal Albanese dimension are non-uniruled. The results in this paper can be applied to study abundance for 3-folds with $\dim \mathrm{Pic}^0(X) >0$, which is finally proved in a later paper \cite{Zh17}.
\end{Remark}

\textbf{Strategy of the proof:} Let's explain our idea to study subadditivity of Kodaira dimensions. Recall that by the standard approach proposed by Viehweg in \cite{Vie}, granted the bigness of $K_Y$, we only need to prove the weak positivity of $f_*\omega_{X/Y}^l$. Unfortunately, in positive characteristics, if fibers have bad singularities, then the sheaf $f_*\omega_{X/Y}^l$ is not necessarily weakly positive (see Raynaud's example \ref{ray-example} below). To overcome this difficulty, stimulated by \cite{PSZ} and \cite{Pa2}, we prove a positivity result (Theorem \ref{mthp} below) without singularity conditions, but at the cost of assuming other conditions like nefness and relative semi-ampleness. These conditions are closely related to minimal model theory. For a fibration of a 3-fold, by passing to a minimal model, we can prove that the sheaf $F_Y^{g*}f_*(\omega_{X/Y}^l \otimes f^*\omega_Y^{l-1})$ contains a non-zero weakly positive subsheaf under certain situations (say, when $\omega_{X/Y}$ is $f$-big), which plays a similar role as the sheaf $f_*\omega_{X/Y}^l$.

The positivity result mentioned above is stated as follows.
\begin{Theorem}\label{mthp}
Let $f:X\rightarrow Y$ be a separable surjective projective morphism between two normal projective varieties over an algebraically closed field $k$ with $\mathrm{char}~k = p > 0$. Assume that $Y$ is Gorenstein. Let $\Delta$ be an effective $\mathbb{Q}$-Weil divisor on $X$ such that $K_{X/Y}+ \Delta$ is $\mathbb{Q}$-Cartier and $p \nmid \mathrm{ind}((K_{X/Y} + \Delta)_\eta)$. If $D$ is a Cartier divisor on $X$ such that $D - K_{X/Y} - \Delta$ is nef and $f$-semi-ample, then
for sufficiently divisible $g$, the sheaf $F_Y^{g*}f_*\mathcal{O}_X(D)$ contains a weakly positive subsheaf $S_{\Delta}^{g}f_*\mathcal{O}_X(D)$ of rank $\dim_{k(\bar{\eta})} S_{\Delta_{\bar{\eta}}}^0(X_{\bar{\eta}}, D_{\bar{\eta}})$.

Moreover if $Y$ is smooth, then $t(Y, S_{\Delta}^{g}f_*\mathcal{O}_X(D), H) \geq 0$ for an ample divisor $H$ on $Y$.
\end{Theorem}

\begin{Remark}\label{rmk-of-positivity}
(1) Please refer to Sec. \ref{Ftrm} and \ref{wp} for the definitions of $S_{\Delta}^{g}f_*\mathcal{O}_X(D)$, $S_{\Delta_{\bar{\eta}}}^0(X_{\bar{\eta}}, D_{\bar{\eta}})$ and $t(Y, S_{\Delta}^{g}f_*\mathcal{O}_X(D), H)$. The invariant $t(Y, \mathcal{F}, H)$ for a coherent sheaf $\mathcal{F}$ was introduced by Ejiri in \cite{Ej} to measure the positivity of $\mathcal{F}$. For example, the condition $t(Y, \mathcal{F}, H) \geq 0$ implies the weak positivity of $\mathcal{F}$, and they are equivalent when $Y$ is a curve. In positive characteristic, to construct global sections, we will use the condition $t(Y, \mathcal{F}, H) \geq 0$ instead of weak positivity (Theorem \ref{F-p-subadd-of-kod-dim}).

(2) In \cite[Theorem D and Theorem E]{PSZ}, the authors got similar results under stronger conditions that $f$ is flat, relatively $G_1$ and $S_2$, $p \nmid \mathrm{ind}(K_{X/Y}+ \Delta)$ and $D - K_{X/Y} - \Delta$ is nef and $f$-ample. And in \cite[Sec. 6]{Pa2}, Patakfalvi proved the weak positivity of $S^{g}f_*\omega_{X/Y}$ under some mild assumptions.
The idea of the proof is to consider the trace maps of relative Frobenius iterations, similarly as in \cite{PSZ} and \cite{Ej}.
\end{Remark}

Applying the theorem above to log minimal models, immediately we get
\begin{Corollary}
Let $f:X\rightarrow Y$ be a separable surjective projective morphism between two normal projective varieties over an algebraically closed field $k$ with $\mathrm{char}~k =p > 0$. Let $\Delta$ be an effective $\mathbb{Q}$-Weil divisor on $X$ such that $K_X + \Delta$ is $\mathbb{Q}$-Cartier and $p \nmid \mathrm{ind} (K_X + \Delta)_\eta$. Assume that $K_X + \Delta$ is nef and $f$-semi-ample and $Y$ is Gorenstein. Then for a positive integer $l$ such that $l(K_{X} + \Delta)$ is Cartier and sufficiently divisible $g$, the sheaf $F_Y^{g*}(\mathcal{O}_X(l(K_{X/Y} + \Delta)) \otimes f^*\omega_Y^{l-1})$ contains a weakly positive subsheaf of rank $\dim_{k(\bar{\eta})} S_{\Delta_{\bar{\eta}}}^0(X_{\bar{\eta}}, l(K_{X/Y} + \Delta)_{\bar{\eta}})$.
\end{Corollary}

Let's recall Raynaud's example, which gives a minimal surface $S$ of general type over a curve $C$ such that for $l \geq 2$, $f_*\omega_{S/C}^l$ is not nef while $(f_*\omega_{S/C}^l) \otimes \omega_C^{l-1}$ is nef.
\begin{Example}[{\cite{Ra}, \cite[Theorem 3.6]{Xie}}]\label{ray-example}
Let $k$ be an algebraically closed field with $\mathrm{char}~k = p\geq 3$. We can find a \emph{Tango curve} $C$ of \emph{integral type} over $k$ of genus $g\geq 2$ (cf. \cite[Ex. 2.4, Def. 2.6]{Xie}), namely, $C$ has a divisor $L=[\frac{df}{p}]$ for some $f \in K(C)$ such that $\deg L >0$ and $L':=\frac{L}{2}$ is integral. By abusing notation, we also use $L,L'$ to denote the corresponding line bundles. Notice that $L$ is a sub-line bundle of $\mathcal{B}^1$ \emph{($=d \mathcal{O}_C \subset \Omega_C^1$)}, which implies $0< h^0(\mathcal{B}^1(-L)) \leq h^0((F_{C*}\Omega_C^1)\otimes L^{-1})) = h^0(\Omega_C^1\otimes L^{-p}))$, hence $\deg(K_C-pL) \geq 0$.
We have a non-trivial extension
$$0 \rightarrow \mathcal{O}_C \rightarrow \mathcal{E} \rightarrow L \rightarrow 0$$
such that $\mathrm{Sym}^p\mathcal{E} \otimes L^{-p}$ has a non-zero section.

Let $X = \mathbb{P}_C(\mathcal{E}^*) = \mathrm{Proj}_{\mathcal{O}_C} \oplus_l \mathrm{Sym}^l\mathcal{E}$, $g: X \rightarrow C$ the natural projection, $E$ the natural section such that $E \sim \mathcal{O}_X(1)$ and $C'$ a smooth curve on $X$ such that $C' \sim pE - pf^*L$. Then $E$ and $C'$ are disjoint to each other. Let
$$M \sim \frac{p+1}{2}E - p f^*L'.$$
Denote by $\pi: S \rightarrow X$ the smooth double cover induced by the relation $2M \sim E + C'$, and by $f: S \rightarrow C$ the natural fibration. Then we have that
$$\pi_*\omega_{S/C}^l \cong \mathcal{O}_X(l(K_{X/C} + M)) \oplus \mathcal{O}_X(lK_{X/C} + (l-1)M),$$
thus by $K_{X/C} \sim -2E + g^*\det \mathcal{E} \sim -2E + g^*L$,
\begin{align*}
f_*\omega_{S/C}^l \cong &g_*\mathcal{O}_X(l(K_{X/C} + M)) \oplus \mathcal{O}_X(lK_{X/C} + (l-1)M)) \\
\cong &g_*(\mathcal{O}_X(\frac{l(p-3)}{2}E + g^*(2-p)lL')) \\
      &\oplus g_*(\mathcal{O}_X(\frac{lp - p - 3l - 1}{2}E + g^*((2-p)l + p)L')) \\
\cong &(\mathrm{Sym}^{\frac{l(p-3)}{2}}\mathcal{E}\otimes L'^{(2-p)l}) \oplus (\mathrm{Sym}^{\frac{lp - p - 3l - 1}{2}}\mathcal{E}\otimes L'^{(2-p)l + p}).
\end{align*}
By easy calculations, one can verify that for every positive integer $l$, the sheaf $f_*\omega_{S/C}^l$ is not nef, instead its dual $(f_*\omega_{S/C}^l)^*$ is nef,
and the sheaf $(f_*\omega_{S/C}^l) \otimes \omega_C^{l-1}$ is nef for $l \geq 2$ since $\deg \omega_C \geq 2p\deg L'$.
\end{Example}

\begin{Acknowledgments}
The author expresses his gratitude to Dr. Sho Ejiri for many useful discussions, and to Prof. Zsolt Patakfalvi and Chenyang Xu for some useful communications. He is very grateful to the anonymous referees for pointing out many inaccuracies and giving valuable suggestions to improve this paper.
The author is supported by grant
NSFC (No. 11771260 and No. 11531009).
\end{Acknowledgments}

\section{Preliminaries}\label{tools}

\subsection{Almost Cartier divisors}
Because the fibrations we treat may have non-normal geometric generic fibers, their Frobenius base changes are not necessarily normal. In the following we will often work on non-normal varieties satisfying Serre condition $S_2$. On $S_2$ varieties, we will work with almost Cartier divisors. Let's recall the definition and some basic results of almost Cartier divisors. Please refer to \cite[Sec. 2.1]{MS} and \cite[p.171-172]{Ko92} for more details.

\begin{Definition}
Let $X$ be a reduced Noetherian $S_2$ scheme over a field $k$ of finite type and of pure dimension. An \emph{almost Cartier divisor} $($\emph{AC divisor} for short$)$ on $X$ is a reflexive coherent $\mathcal{O}_X$-submodule of the sheaf of total quotient ring $K(X)$ such that invertible in codimension one.
\end{Definition}

We remark the following results.

(1) Recall that if $X$ is normal and $D$ is a Weil divisor on $X$, the sheaf $\mathcal{O}_X(D)$ is defined via
$$\mathcal{O}_X(D)_x:=\{f \in K(X)| ((f) + D)|_U \geq 0 ~\mathrm{for~some~open~set}~U~\mathrm{containing}~x\},$$
hence is an AC divisor. In the normal setting, we have a natural one-to-one correspondence between the set of Weil divisors and the set of AC divisors.

(2) To give an AC divisor, it is equivalent to give a Cartier divisor $D_0$ on an open subset $X_0$ which is the complement of a closed subscheme $S$ with $\mathrm{codim}_X S \geq 2$. Indeed, denoting by $i: X_0 \hookrightarrow X$ the natural inclusion, the sheaf $i_*\mathcal{O}_{X_0}(D_0)$ is an AC divisor (\cite[p.172]{Ko92}).

(3) Denote by $\mathrm{WSh}(X)$ the set of AC divisors, which is an additive group.
For $D \in \mathrm{WSh}(X)$, in the following to unify the notation, we use the notation $\mathcal{O}_X(D)$ for the coherent sheaf defining $D$; and we say $D$ is effective, if $\mathcal{O}_X \subseteq \mathcal{O}_X(D)$. For two AC divisors $D_1$ and $D_2$ on $X$, we denote $D_1 \geq D_2$ if $E= D_2 - D_1$ is effective.
An element of $\mathrm{WSh}(X) \otimes \mathbb{Q}$ is called a $\mathbb{Q}$-AC divisor.

(4) We can define the linear and $\mathbb{Q}$-linear equivalences in $\mathrm{WSh}(X)$ and $\mathrm{WSh}(X) \otimes \mathbb{Q}$, which are denoted by $\sim$ and $\sim_{\mathbb{Q}}$ respectively. And for a morphism $g: Z \to X$ of reduced Noetherian $S_2$ schemes, we can always define the pullback of $\mathbb{Q}$-Cartier $\mathbb{Q}$-divisors, and if $g$ is equi-dimensional we can define the pullback of $\mathbb{Q}$-AC divisors.

\subsection{Relative canonical sheaves}
Let $X$ be a reduced, $G_1$ and $S_2$ projective scheme of pure dimension and of finite type over a field $k$. The \emph{canonical sheaf} $\omega_X$ of $X$ is defined as its dualizing sheaf (\cite[V.8-10, VI.2]{Ha}, \cite[Sec. 5.5]{KM98}). Then over the Gorenstein open set $i: X_0 \hookrightarrow X$ such that $\mathrm{codim}_X (X \setminus X_0) \geq 2$, the restriction $\omega_{X}|_{X_0}$ is a line bundle, the \emph{canonical divisor} (class) $K_X$ is the equivalent class of AC divisors satisfying that $\mathcal{O}_X(K_{X})$ is isomorphic to $\omega_{X}$.

Let $f: X \to Y$ be a projective morphism between reduced, $G_1$ and $S_2$ Noetherian schemes of pure dimension and of finite type over a field $k$. If either $Y$ is Gorenstein, or $f$ is flat then similarly as above we can define the \emph{relative canonical sheaf} $\omega_{X/Y}$ and the \emph{relative canonical divisor} $K_{X/Y}$.

It is known that relative canonical sheaves are compatible with flat base changes (cf. \cite[Chap. III Sec. 8]{Ha}). To treat non-flat base changes, we have the following result which is similar to \cite[Theorem 2.4]{CZ}.
\begin{Proposition}\label{compds}
Let $f: X \rightarrow Y$ be a separable flat projective morphism between two normal varieties. Let $\Delta$ be an effective $\mathbb{Q}$-Weil divisor on $X$ such that $K_{X/Y} + \Delta$ is $\mathbb{Q}$-Cartier.
Let $\pi: Y' \rightarrow Y$ be a generically finite surjective morphism from a smooth variety, $\bar{X}'= X\times_Y Y'$ and $\sigma: X' \rightarrow \bar{X}$ the normalization morphism, which fit into the following commutative diagram
\begin{center}
\xymatrix@C=2cm{
&X' \ar@/^2pc/[rr]|{\sigma'}\ar[r]^>>>>>>>>>{\sigma} \ar[rd]^{f'} &\bar{X}' = X \times_Y Y'
\ar[r]^<<<<<<<<<{\pi'}\ar[d]^{\bar{f}'}    &X\ar[d]^f \\
& &Y'\ar[r]^{\pi}   &Y
}
\end{center}
where $\pi'$ and $\bar{f}'$ denote the natural projections, and $f'=\bar{f}'\circ \sigma$.

Then there exist an effective $\sigma$-exceptional Cartier divisor $E'$ coming from the pullback of a divisor on $Y'$, and an effective divisor $\Delta'$ on $X'$ such that
$$K_{X'/Y'} + \Delta' = \sigma'^*(K_{X/Y} + \Delta) + E'.$$
\end{Proposition}
\begin{proof}
Denote by $Y_0$ the smooth locus of $Y$, and let $Y'_0 = \pi^{-1}Y_0$, $X_0 = X\times_Y Y_0$, $\bar{X}'_0 = \bar{X}'\times_{Y'}Y'_0$ and $X'_0 = X'\times_{Y'}Y'_0$. By arguing in codimension one, we assume $X_0$ is Gorenstein, hence $f|_{X_0}$ is a flat Gorenstein morphism by \cite[p.298 (Ex. 9.7)]{Ha}.
Then by remarks of \cite[p.388]{Ha}, we have
$$K_{\bar{X}'_0/Y'_0} \sim_{\mathbb{Q}} \pi'^*K_{X/Y}|_{\bar{X}'_0}.$$
Since $X'_0 \rightarrow \bar{X}'_0$ is the normalization, by results of \cite[Sec. 2]{Re}, there exists an effective divisor $C'$ such that
$$\sigma'^*K_{X/Y}|_{\bar{X}'_0} \sim_{\mathbb{Q}} \sigma^*K_{\bar{X}'_0/Y'_0} = K_{X'_0/Y'_0} + C'.$$
Since $K_{X/Y} + \Delta$ is assumed to be $\mathbb{Q}$-Cartier, its pull-back makes sense. In turn we can get an effective divisor $\Delta'_0$ on $X'_0$ such that
$$K_{X'_0/Y'_0} + \Delta'_0 \sim_{\mathbb{Q}} \sigma'^*(K_{X/Y} + \Delta)|_{\bar{X}'_0}.$$

Let $D'$ be the closure of $\Delta'_0$ in $X'$, which is a $\mathbb{Q}$-Weil divisor. Let $B'= \sigma'^*(K_{X/Y} + \Delta) - (K_{X'/Y'} + D')$. If $B'= 0$, then we are done. Otherwise, since $f'$ is equi-dimensional, the support of $B'$ is mapped via $f'$ to a codimension one cycle contained in $Y' \setminus Y'_0$. Since $Y'$ is smooth, we can find an effective $\pi$-exceptional Cartier divisor $E$ on $Y'$ such that $D''=f'^*E + B' \geq 0$. Then we are done by setting $E'= f'^*E$ which is $\sigma'$-exceptional, and $\Delta'= D' + D''$.
\end{proof}

\subsection{Trace maps of Frobenius iterations} \label{Ftrm}
Throughout this subsection, let $k$ be an algebraically closed field of characteristic $p > 0$.
Let $f: X \rightarrow Y$ be a morphism of schemes over $k$. We will use the following notation:

(1) $F_X^e: X \rightarrow X$ for the $e^{\mathrm{th}}$ absolute Frobenius iteration, and sometimes, to avoid confusions, we use $X^e$ for the source scheme in the morphism $F_X^e: X \rightarrow X$;

(2) $X_{Y^e}$ for the fiber product $X\times_Y Y^e$ of morphisms $f: X \rightarrow Y$ and $F_Y^e: Y^e \rightarrow Y$, $f_e: X_{Y^e} \rightarrow Y$ and $\pi_Y^e: X_{Y^e} \rightarrow X$ for the natural projections;

(3) $F_{X/Y}^e: X \rightarrow X_{Y^e}$ for the $e^{\mathrm{th}}$ relative Frobenius iteration over $Y$.

We will discuss the trace maps of (relative) Frobenius iterations in different settings. Please refer to \cite{Pa}, \cite{Pa2}, \cite{PSZ} and \cite{Ej} for more details and related results.

\subsubsection{Trace maps of absolute Frobenius iterations}\label{tmaF}
\begin{Notation}\label{1} Let $X$ be a reduced, $G_1$ and $S_2$ projective scheme over $k$ of finite type and of pure dimension. Denote by $X_0$ a Gorenstein open subset of $X$ such that $\mathrm{codim}_X (X\backslash X_0) >1$.
Let $\Delta$ be an effective $\mathbb{Q}$-AC divisor such that $K_X + \Delta$ is $\mathbb{Q}$-Cartier.
Assume the Cartier index $\mathrm{ind} (K_X + \Delta)$ is indivisible by $p$. Then there exists a positive integer $g$ such that $(1-p^{eg})(K_X + \Delta)$ is Cartier for every positive integer $e$, in particular $(p^{eg}-1)\Delta|_{X_0}$ is an effective Cartier divisor. Let $D$ be a Cartier divisor on $X$.
\end{Notation}
\begin{Remark}
To give a divisor $\Delta$ as above is equivalent to give a sub-line bundle $\mathcal{L}_g$ of $\mathcal{O}_X ((1-p^{g})K_X)$ for some $g$. Indeed, from $\Delta$ we can immediately get the sub-line bundle $\mathcal{O}_X ((1-p^{g})(K_X + \Delta))$ of $\mathcal{O}_X ((1-p^{g})K_X)$; conversely, given a sub-line bundle $\mathcal{L}_g \subseteq \mathcal{O}_X ((1-p^{g})K_X)$, assuming this inclusion is induced by an effective AC divisor $B$, we get the $\mathbb{Q}$-AC divisor $\Delta = \frac{B}{p^{g} - 1}$ which satisfies the assumptions in Notation \ref{1}.
\end{Remark}

Since $X$ is $G_1$ and $S_2$, the composite map of the natural inclusion
$$F^{eg}_{X*} \mathcal{O}_X ((1-p^{eg})(K_X + \Delta))|_{X_0} \hookrightarrow F^{eg}_{X_0*} \mathcal{O}_{X_0} ((1-p^{eg})K_{X_0})$$
and the trace map $Tr_{X_0}^{eg}: F^{eg}_{X_0*} \mathcal{O}_{X_0} ((1-p^{eg})K_{X_0}) \cong F^{eg}_{X_0*} \omega_{X_0}^{1-p^{eg}}\rightarrow  \mathcal{O}_{X_0}$ extends to a map on $X$:
$$Tr_{X,\Delta}^{eg}: F^{eg}_{X*} \mathcal{O}_X ((1-p^{eg})(K_X + \Delta)) \rightarrow  \mathcal{O}_X.$$
Twisting the trace map $Tr_{X,\Delta}^{eg}$ above by $\mathcal{O}_X(D)$ induces a map
\begin{align*}
Tr_{X,\Delta}^{eg}(D): &F^{eg}_{X*} \mathcal{O}_X ((1-p^{eg})(K_X + \Delta)) \otimes \mathcal{O}_X(D) \\
                     &\cong F^{eg}_{X*} \mathcal{O}_X ((1-p^{eg})(K_X + \Delta)+ p^{eg}D)  \rightarrow  \mathcal{O}_X(D),
\end{align*}
then taking global sections gives
$$H^0(Tr_{X,\Delta}^{eg}(D)): H^0(X, F^{eg}_{X*} \mathcal{O}_X ((1-p^{eg})(K_X + \Delta)+ p^{eg}D)) \rightarrow  H^0(X, D).$$
Let
$$S_{\Delta}^{eg}(X, D) = \mathrm{Im} H^0(Tr_{X,\Delta}^{eg}(D)) ~ \mathrm{and} ~S_{\Delta}^{0}(X, D) = \cap_{e\geq 0} S_{\Delta}^{eg}(X, D).$$
If $\Delta = 0$, we usually use the notation $S^{0}(X, D)$ instead of $S_{0}^{0}(X, D)$.

For $e' > e$, the map $Tr_{X,\Delta}^{e'g}(D)$ factors as
\begin{align*}
Tr_{X,\Delta}^{e'g}(D):~
&F^{eg}_{X*} F^{(e'-e)g}_{X*}\mathcal{O}_X ((1-p^{e'g})(K_X + \Delta)+ p^{e'g}D) \\
&\xrightarrow{F^{eg}_{X*}Tr_{X,\Delta}^{(e'-e)g}((1-p^{eg})(K_X + \Delta)+ p^{eg}D)} F^{eg}_{X*} \mathcal{O}_X ((1-p^{eg})(K_X + \Delta)+ p^{eg}D)  \\
&\xrightarrow{Tr_{X,\Delta}^{eg}(D)} \mathcal{O}_X(D).
\end{align*}
So there is a natural inclusion $S_{\Delta}^{e'g}(X, D) \subseteq S_{\Delta}^{eg}(X, D)$, thus for sufficiently large $e$, $S_{\Delta}^{eg}(X, D)= S_{\Delta}^{0}(X, D)$.

\begin{Proposition}\label{F-non-vanishing}
Let the notation be as in Notation \ref{1}. Then

(1) There exists an ideal $\sigma(X, \Delta)$, namely, the non-$F$-pure ideal of $(X, \Delta)$, such that for sufficiently divisible $e$,
$$\mathrm{Im} Tr_{X,\Delta}^{eg}  = \sigma(X, \Delta) = Tr_{X,\Delta}^{eg} (F^{eg}_{X*} (\sigma(X, \Delta)\cdot \mathcal{O}_X ((1-p^{eg})(K_X + \Delta)))).$$

(2) If $D$ is ample, then for sufficiently large $l$
$$S_{\Delta}^{0}(X, lD) = H^0(X, \sigma(X, \Delta) \cdot \mathcal{O}_X(lD)).$$

(3) Assume $X$ is integral and $D$ is big, and let $E$ be another Cartier divisor on $X$. Then for sufficiently large $l$, $S_{\Delta}^{0}(X, lD + E) \neq 0$.

(4) Assume that $X$ is integral.
Let $\sigma: X' \rightarrow X$ be the normalization morphism. Let $K_{X}$ be a fixed canonical divisor of $X$, and let $C$ denote the divisor arising from the conductor of the normalization, which gives the canonical divisor $K_{X'} = \sigma^*K_X -B$ $($\cite[2.6]{Re}$)$, and let $\Delta'  = \sigma^* \Delta + B$ $($namely, $K_{X'} + \Delta'= \sigma^*(K_X + \Delta)$$)$. Then
$$\dim S_{\Delta}^{0}(X, D) \leq \dim S_{\Delta'}^{0}(X', \sigma^*D).$$
\end{Proposition}
\begin{proof}
For (1), please refer to \cite[Lemma 13.1]{Ga}.
\smallskip

(2) Fix a sufficiently divisible $g$ such that for every integer $e >0$ the trace map below is surjective
$$Tr_{X,\Delta}^{eg}:~F^{eg}_{X*} (\sigma(X, \Delta)\cdot \mathcal{O}_X ((1-p^{eg})(K_X + \Delta))) \rightarrow \sigma(X, \Delta),$$
and denote its kernel by $\mathcal{B}^{eg}$. Then we have the following exact sequence
{\small \begin{align*}
0 \to F^{(e-1)g}_{X*} (\mathcal{B}^{g} \otimes &\mathcal{O}_X ((1-p^{(e-1)g})(K_X + \Delta)))
\to F^{eg}_{X*} (\sigma(X, \Delta)\cdot \mathcal{O}_X ((1-p^{eg})(K_X + \Delta)))   \\
&\to F^{(e-1)g}_{X*} (\sigma(X, \Delta)\cdot \mathcal{O}_X ((1-p^{(e-1)g})(K_X + \Delta))) \to 0.
\end{align*}}
Mimicking the proof of \cite[Lemma 2.20]{Pa}), we can obtain another exact sequence
\begin{align*}
0\to F^{(e - 1)g}_{X*}(\mathcal{B}^{g} \otimes \mathcal{O}_X ((1-p^{(e - 1)g})(K_X + \Delta))) \to  \mathcal{B}^{eg} \to  \mathcal{B}^{(e - 1)g} \to 0.
\end{align*}
We can find an integer $l_0$ such that $l_0D - (K_X + \Delta)$ is ample, and applying Fujita vanishing (\cite{Keeler}),
in turn we can find $l_1>l_0$ such that for every integer $l > l_1$ and $e \geq 1$
\begin{align*}
&H^1(X, F^{(e - 1)g}_{X*}(\mathcal{B}^{g}\otimes \mathcal{O}_X ((1-p^{(e - 1)g})(K_X + \Delta))) \otimes \mathcal{O}_X(lD)) \\
&\cong H^1(X, \mathcal{B}^{g}\otimes \mathcal{O}_X (lD + (p^{(e - 1)g} -1)(lD - (K_X + \Delta)))) = 0.
\end{align*}
By induction on $e$ we can show $H^1(X, \mathcal{B}^{eg} \otimes \mathcal{O}_X(lD)) = 0$ for every $e>0$, which implies
\begin{align*}
H^0(Tr_{X,\Delta}^{eg}(lD)):~&H^0(X, F^{eg}_{X*} (\sigma(X, \Delta)\cdot \mathcal{O}_X ((1-p^{eg})(K_X + \Delta)))\otimes \mathcal{O}_X(lD))\\
&\rightarrow H^0(X, \sigma(X, \Delta)\cdot \mathcal{O}_X(lD))
\end{align*}
is surjective.
\smallskip

(3) By the result (2) we can take a sufficiently ample divisor $H$ such that $S_{\Delta}^{0}(X, H) \neq 0$. Since $D$ is big, for sufficiently large integer $l$
$$H^0(X, \mathcal{O}_X(lD + E - H)) \neq 0.$$
Take a nonzero section $s_F \in H^0(X, \mathcal{O}_X(lD + E - H))$ which defines a divisor $F \in |lD + E - H|$.
We have the natural maps
\begin{align*}
&H^0(X, F^{g}_{X*} \mathcal{O}_X ((1-p^{g})(K_X + \Delta))\otimes \mathcal{O}_X(H))\\
&\xrightarrow{\otimes s_F} H^0(X, F^{g}_{X*} \mathcal{O}_X ((1-p^{g})(K_X + \Delta))\otimes \mathcal{O}_X(H + F))\\
&\xrightarrow{H^0(Tr_{X,\Delta}^{g}(H + F))} H^0(X, \mathcal{O}_X(H + F)) = H^0(X, \mathcal{O}_X(lD + E)).
\end{align*}
Then we get an injection $S_{\Delta}^{0}(X, H)\xrightarrow{\otimes s_F} S_{\Delta}^{0}(X, lD + E)$, which finishes the proof of (3).
\smallskip

We are left to prove (4). To ease the notation, we will use $\mathcal{O}_X^{\frac{1}{p^g}}$ to denote the structure sheaf of $X^g$; and for an AC divisor $D$ on $X^g$, the pushforward $F_{X*}^{g}\mathcal{O}_{X^g}(D)$ is granted an $\mathcal{O}_X^{\frac{1}{p^g}}$-module structure and will be denoted by $\mathcal{O}_{X}(D)^{\frac{1}{p^g}}$.

By duality theory (\cite[Sec. III.6]{Ha}), we have an $\mathcal{O}_{X} ^{\frac{1}{p^g}}$-linear isomorphism
$$\mathcal{O}_X((1-p^g)K_X)^{\frac{1}{p^g}} ~\cong ~\mathcal{H}om_{\mathcal{O}_{X}}(\mathcal{O}_{X}^{\frac{1}{p^g}}, \mathcal{O}_{X}),$$
and the trace map $Tr_{X, \Delta}^g$ is give by the composition
\begin{align*}
Tr_{X, \Delta}^g: \mathcal{O}_X((1-p^g)(K_X + \Delta))^{\frac{1}{p^g}} ~\cong ~&\mathcal{H}om_{\mathcal{O}_{X}}(\mathcal{O}_{X}((p^g -1)\Delta)^{\frac{1}{p^g}}, \mathcal{O}_{X}) \\
& \subseteq ~\mathcal{H}om_{\mathcal{O}_{X}}(\mathcal{O}_{X}^{\frac{1}{p^g}}, \mathcal{O}_{X}) \xrightarrow{\mathrm{ev}(1)} \mathcal{O}_{X}
\end{align*}
where $\mathrm{ev}(1)$ denotes the evaluation map at $1$.

Let $K = K(X') = K(X)$. Regard
{\small $\mathcal{H}om_{\mathcal{O}_{X}}(\mathcal{O}_{X}((p^g -1)\Delta)^{\frac{1}{p^g}}, \mathcal{O}_{X})\otimes_{\mathcal{O}_X^{\frac{1}{p^g}}} \mathcal{O}_{X'}^{\frac{1}{p^g}}$} and $\mathcal{H}om_{\mathcal{O}_{X'}}(\mathcal{O}_{X'}((p^g -1)\Delta')^{\frac{1}{p^g}}, \mathcal{O}_{X'})$ as two sub-sheaves of the constant sheaf $\mathrm{Hom}_K(K^{\frac{1}{p^g}}, K) \cong K^{\frac{1}{p^g}}$ on $X'^g$, then both are line bundles linearly equivalent to the divisor $\sigma^*(1-p^g)(K_{X} + \Delta) = (1-p^g)(K_{X'} + \Delta')$. We claim that
{\small $$\mathcal{H}om_{\mathcal{O}_{X}}(\mathcal{O}_{X}((p^g -1)\Delta)^{\frac{1}{p^g}}, \mathcal{O}_{X})\otimes_{\mathcal{O}_X^{\frac{1}{p^g}}} \mathcal{O}_{X'}^{\frac{1}{p^g}} = \mathcal{H}om_{\mathcal{O}_{X'}}(\mathcal{O}_{X'}((p^g -1)\Delta')^{\frac{1}{p^g}}, \mathcal{O}_{X'}).$$}
It suffices to verify this claim in codimension one. So we may assume both $X$ and $X'$ are Gorenstein. Take an open affine subset $U =\mathrm{Spec} R \subseteq X$. Let $R^N$ denote the normalization of $R$. By abusing notation we still use $\Delta$ to denote the restriction $\Delta|_U$. Since $\mathcal{H}om_{\mathcal{O}_{X}}(\mathcal{O}_{X}((p^g -1)\Delta)^{\frac{1}{p^g}}, \mathcal{O}_{X})(U) = \mathrm{Hom}_{R}(R((p^g-1)\Delta)^{\frac{1}{p^g}}, R)$ is a free $R^{\frac{1}{p^g}}$-module of rank one, we can take a generator $\phi$, which corresponds to the divisor $\Delta_{\phi} = \Delta$ via the correspondence of \cite[Theorem 2.4]{MS}. By \cite[Lemma 3.1]{MS} (or \cite[Prop. 7.10, 7.11]{Sch10}), $\phi$ extends to an element $\sigma^*\phi \in  \mathrm{Hom}_{R^N}((R^N)^{\frac{1}{p^g}}, R^N)$, which corresponds to the divisor $\Delta_{\sigma^*\phi} =\sigma^*\Delta_{\phi} + B = \Delta'$ by
\cite[Lemma 3.1]{MS}\footnote{Remark that in the statement of \cite[Lemma 3.1]{MS}, $R$ is required to be semi-normal, but this condition is not used in the proof.}. This means
$\sigma^*\phi$ is a generator of the $(R^N)^{\frac{1}{p^g}}$-module $\mathrm{Hom}_{R^N}(R^N((p^g-1)\Delta')^{\frac{1}{p^g}}, R^N)$. Then we can conclude the claim.

Applying the claim above, we can extend $Tr_{X, \Delta}^g$ to an $\mathcal{O}_{X'}$-linear map
{\small \begin{align*}
\sigma^*Tr_{X, \Delta}^g: ~&\mathcal{O}_{X'}((1-p^g)(K_{X'} + \Delta'))^{\frac{1}{p^g}} \\
&\cong~\sigma^*\mathcal{O}_X((1-p^g)(K_X + \Delta))^{\frac{1}{p^g}}\\
~&\cong~ \mathcal{H}om_{\mathcal{O}_{X}}(\mathcal{O}_{X}((p^g -1)\Delta)^{\frac{1}{p^g}}, \mathcal{O}_{X})\otimes_{\mathcal{O}_X^{\frac{1}{p^g}}} \mathcal{O}_{X'}^{\frac{1}{p^g}} \\
&=~ \mathcal{H}om_{\mathcal{O}_{X'}}(\mathcal{O}_{X'}((p^g -1)\Delta')^{\frac{1}{p^g}}, \mathcal{O}_{X'}) \subseteq \mathcal{H}om_{\mathcal{O}_{X'}}(\mathcal{O}_{X'}^{\frac{1}{p^g}}, \mathcal{O}_{X'}) \xrightarrow{\mathrm{ev}(1)} \mathcal{O}_{X'}.
\end{align*}}
By the above construction, since $X'$ is integral and projective, we see that $\sigma^*Tr_{X, \Delta}^g$ coincides with $Tr_{X', \Delta'}^g$ up to some multiplication by a nonzero number of $k$.
From this we can verify the following commutative diagram
$$\xymatrix@C=1cm{
&H^0(\mathcal{O}_X((1-p^g)(K_X + \Delta))^{\frac{1}{p^g}}\otimes_{\mathcal{O}_{X}}\mathcal{O}_{X}(D))\ar[r]^<<<<<<<<<{Tr_{X, \Delta}^g}\ar[d]^{\sigma^*}     &H^0(\mathcal{O}_{X}(D))\ar[d]^{\sigma^*}\\
&H^0(\mathcal{O}_{X'}((1-p^g)(K_{X'} + \Delta'))^{\frac{1}{p^g}}\otimes_{\mathcal{O}_{X'}}\mathcal{O}_{X'}(\sigma^*D ))\ar[r]^<<<<{Tr_{X', \Delta'}^g} &H^0(\mathcal{O}_{X'}(\sigma^*D))
}$$
In turn we get an injection $\sigma^*: S_{\Delta}^{0}(X, D) \hookrightarrow S_{\Delta'}^{0}(X', \sigma^*D)$, which implies that $\dim S_{\Delta}^{0}(X, D) \leq \dim S_{\Delta'}^{0}(X', \sigma^*D)$.
\end{proof}

\subsubsection{Trace maps of relative Frobenius iterations I}\label{tmrFI}
\begin{Notation}\label{2}
Let $f: X \rightarrow Y$ be a separable surjective projective morphism between two schemes over $k$ of finite type and of pure dimension.
Assume that $X$ is reduced, $G_1$ and $S_2$ and that $Y$ is integral and regular. Let $\Delta$, $g$ and $D$ be assumed as in Notation \ref{1}.
\end{Notation}
By assumption $F_Y^{eg}$ is a flat morphism, so $X_{Y^{eg}}$ also satisfies $G_1$ and $S_2$, and $K_{X_{Y^{eg}}/Y^{eg}} = \pi_Y^{eg*}K_{X/Y}$. By easy calculation we have that
$$K_{X^{eg}/X_{Y^{eg}}} = (1-p^{eg})K_{X^{eg}/Y^{eg}} ~\mathrm{and}~ F_{X/Y}^{eg*}\pi_Y^{eg*}D = p^{eg}D.$$
Similarly as in \ref{tmaF}, we get the trace map
$$
Tr_{X/Y,\Delta}^{eg}(D): F_{X/Y*}^{eg}\mathcal{O}_X ((1-p^{eg})(K_{X/Y} + \Delta)+ p^{eg}D) \rightarrow \mathcal{O}_{X_{Y^{eg}}}(\pi_Y^{eg*}D).
$$
Applying $f_{eg*}$ to the above map, we get
\begin{align*}
f_*Tr_{X/Y,\Delta}^{eg}(D): f_*\mathcal{O}_X &((1-p^{eg})(K_{X/Y} + \Delta)+ p^{eg}D) \\
&\twoheadrightarrow S_{\Delta}^{eg}f_*\mathcal{O}_X(D) \hookrightarrow f_{eg*}\mathcal{O}_{X_{Y^{eg}}}(\pi_Y^{eg*}D) \cong F_Y^{eg*}f_*\mathcal{O}_X(D).
\end{align*}
where $S_{\Delta}^{eg}f_*\mathcal{O}_X(D)$, introduced in \cite[Def. 6.4]{PSZ} with slightly different notation, denotes the image of $f_*Tr_{X/Y,\Delta}^{eg}(D)$. If $\Delta = 0$, we use the notation $S^{eg}f_*\mathcal{O}_X(D)$ instead of $S_{0}^{eg}f_*\mathcal{O}_X(D)$. For $e'>e$, according to the following commutative diagram
$$\xymatrix@=1.5cm{&X\ar[dr]^{F_X^{(e'-e)g}}\ar[d]|{F_{X/Y}^{(e'-e)g}}\ar@/^2.8pc/[rrdd]|{F_X^{e'g}}\ar@/_2.3pc/[ddd]_{f}  & &\\
&X_{Y^{(e'-e)g}}\ar[r]^{\pi_Y^{(e'-e)g}}\ar[dd]  &X\ar[dr]^{F_X^{eg}}\ar[d]_{F_{X/Y}^{eg}}   &\\
& &X_{Y^{eg}}\ar[r]^{\pi_Y^{eg}}\ar[d]^{f_{eg}}   &X\ar[d]^f \\
&Y\ar[r]^{F_Y^{(e'-e)g}}  &Y\ar[r]^{F_Y^{eg}}    &Y
}$$
the trace map $f_*Tr_{X/Y,\Delta}^{e'g}(D)$ factors as
\begin{align*}
f_*\mathcal{O}_X ((1-p^{e'g})(K_{X/Y} + \Delta)+ p^{e'g}&D) \\
\xrightarrow{f_*Tr_{X/Y,\Delta}^{(e'-e)g}((1-p^{eg})(K_{X/Y} + \Delta)+ p^{eg}D)}&F_Y^{(e'-e)g*}f_*\mathcal{O}_X ((1-p^{eg})(K_{X/Y} + \Delta)+ p^{eg}D)\\
 &\xrightarrow{F_Y^{(e'-e)g*}f_*Tr_{X/Y,\Delta}^{eg}(D)}   F_Y^{(e'-e)g*}S_{\Delta}^{eg}f_*\mathcal{O}_X(D).
\end{align*}
Then we conclude a natural inclusion
$$S_{\Delta}^{e'g}f_*\mathcal{O}_X(D) \hookrightarrow F_Y^{(e'-e)g*}S_{\Delta}^{eg}f_*\mathcal{O}_X(D).$$

\begin{Proposition}\label{stable-dim}
Let the notation be as in Notation \ref{2}. Then for every positive integer $e$,
$$\dim_{k(\bar{\eta})}S^{eg}_{\Delta_{\bar{\eta}}}(X_{\bar{\eta}}, D_{\bar{\eta}}) = \mathrm{rank} S_{\Delta}^{eg}f_*\mathcal{O}_X(D).$$
As a consequence, for sufficiently large $e$, $\mathrm{rank}S_{\Delta}^{eg}f_*\mathcal{O}_X(D)= \dim_{k(\bar{\eta})} S^{0}_{\Delta_{\bar{\eta}}}(X_{\bar{\eta}}, D_{\bar{\eta}})$.
\end{Proposition}
\begin{proof}
Consider the following commutative diagram
$$\xymatrix@C=2.5cm{&X_{\bar{\eta}}\ar[d]\ar[r]^<<<<<<<<<<<<<<<<{F_{X_{\bar{\eta}}/\bar{\eta}}^{eg}}\ar@/^2pc/[rr]|{F_{X_{\bar{\eta}}}^{eg}} &X_{\bar{\eta}^{eg}} \ar[d] \ar[r]^>>>>>>>>>>>>>>{\pi_{\bar{\eta}}^{eg}} & X_{\bar{\eta}}\ar[d]\\
&\bar{\eta}\ar@{=}[r]&\bar{\eta} \cong \bar{\eta}^{eg}\ar[r]^{F_{\bar{\eta}}^{eg}} &\bar{\eta}
}.$$
Let $D_{\bar{\eta}^{eg}} = \pi_{\bar{\eta}}^{eg*}D$. Then the trace map w.r.t. the map $\pi_{\bar{\eta}}^{eg}$
$$Tr_{\pi_{\bar{\eta}}^{eg}}: H^0(X_{\bar{\eta}^{eg}}, D_{\bar{\eta}^{eg}}) \rightarrow H^0(X_{\bar{\eta}}, D_{\bar{\eta}})$$
is an isomorphism since $k(\bar{\eta})$ is algebraically closed.
By $F_{X_{\bar{\eta}}}^{eg} = \pi_{\bar{\eta}}^{eg} \circ F_{X_{\bar{\eta}}/\bar{\eta}}^{eg}$, the trace map
$$H^0(Tr_{X_{\bar{\eta}},\Delta_{\bar{\eta}}}^{eg}(D_{\bar{\eta}})): H^0(X_{\bar{\eta}}, F^{eg}_{X_{\bar{\eta}}*} \mathcal{O}_{X_{\bar{\eta}}} ((1-p^{eg})(K_{X_{\bar{\eta}}} + \Delta_{\bar{\eta}})+ p^{eg}D_{\bar{\eta}})) \rightarrow
H^0(X_{\bar{\eta}}, D_{\bar{\eta}})$$
factors as
\begin{align*}
H^0(X_{\bar{\eta}}, &F^{eg}_{X_{\bar{\eta}}*} \mathcal{O}_{X_{\bar{\eta}}} ((1-p^{eg})(K_{X_{\bar{\eta}}} + \Delta_{\bar{\eta}})+ p^{eg}D_{\bar{\eta}}))\\
&\xrightarrow{H^0(Tr_{X_{\bar{\eta}}/\bar{\eta},\Delta_{\bar{\eta}}}^{eg}(D_{\bar{\eta}^{eg}}))} H^0(X_{\bar{\eta}^{eg}}, D_{\bar{\eta}^{eg}}) \xrightarrow{Tr_{\pi_{\bar{\eta}}^{eg}}}H^0(X_{\bar{\eta}}, D_{\bar{\eta}}).
\end{align*}
It follows that
$\dim_{k(\bar{\eta}^{eg})}
\mathrm{Im} H^0(Tr_{X_{\bar{\eta}}/\bar{\eta},\Delta_{\bar{\eta}}}^{eg}(D_{\bar{\eta}^{eg}}))
= \dim_{k(\bar{\eta})}S^{eg}_{\Delta_{\bar{\eta}}}(X_{\bar{\eta}}, D_{\bar{\eta}})$.
On the other hand, since the morphism $i: \bar{\eta} \rightarrow Y$ is flat, the trace map $Tr_{X_{\bar{\eta}}/\bar{\eta},\Delta_{\bar{\eta}}}^{eg}(D_{\bar{\eta}^{eg}})$
coincides with the pull-back map via $i^*$ of
\begin{align*}
Tr_{X/Y,\Delta}^{eg}(D): f_*\mathcal{O}_X ((1-p^{eg})(K_{X/Y} + &\Delta)+ p^{eg}D) \\
&\twoheadrightarrow S_{\Delta}^{eg}f_*\mathcal{O}_X(D) \hookrightarrow  F_Y^{eg*}f_*\mathcal{O}_X(D).
\end{align*}
Then we conclude the proof by
$$\mathrm{rank} S_{\Delta}^{eg}f_*\mathcal{O}_X(D) = \dim_{k(\bar{\eta}^{eg})}
\mathrm{Im} H^0(Tr_{X_{\bar{\eta}}/\bar{\eta},\Delta_{\bar{\eta}}}^{eg}(D_{\bar{\eta}^{eg}}))
= \dim_{k(\bar{\eta})}S^{eg}_{\Delta_{\bar{\eta}}}(X_{\bar{\eta}}, D_{\bar{\eta}}).$$
\end{proof}

\subsubsection{Trace maps of relative Frobenius iterations II: in the normal setting}\label{tmrFII}
\begin{Notation}\label{3}
Let $f: X \rightarrow Y$ be a separable surjective projective morphism between two schemes over $k$ of finite type and of pure dimension. Assume that $X$ is normal and $Y$ is integral and regular. Let $D$ be a Weil divisor on $X$ and $\Delta$ an effective $\mathbb{Q}$-Weil divisor on $X$. Assume that $K_X + \Delta$ is $\mathbb{Q}$-Cartier. It is known that $X_{Y^{e}}$ is $G_1$ and $S_2$, $\pi_Y^{e*}D$ is an AC divisor on $X_{Y^{e}}$, and the sheaf $\mathcal{O}_{X_{Y^{e}}}(\pi_Y^{e*}D)$ is isomorphic to $\pi_Y^{e*}\mathcal{O}_{X}(D)$.
\end{Notation}

Replacing $(p^{eg} - 1)\Delta$ with $[(p^{eg} - 1)\Delta]$, analogously to Sec. \ref{tmrFI}, we get the trace maps
$$
Tr_{X/Y,\Delta}^{eg}(D): F_{X/Y*}^{eg}\mathcal{O}_X ((1-p^{eg})K_{X/Y} - [(p^{eg} - 1)\Delta]+ p^{eg}D) \rightarrow \mathcal{O}_{X_{Y^{eg}}}(\pi_Y^{eg*}D)
$$
and
\begin{align*}
f_*Tr_{X/Y,\Delta}^{eg}(D): f_*\mathcal{O}_X &((1-p^{eg})K_{X/Y} - [(p^{eg} - 1)\Delta] + p^{eg}D) \\
&\twoheadrightarrow S_{\Delta}^{eg}f_*\mathcal{O}_X(D) \hookrightarrow f_{eg*}\mathcal{O}_{X_{Y^{eg}}}(\pi_Y^{eg*}D) \cong F_Y^{eg*}f_*\mathcal{O}_X(D).
\end{align*}
where $S_{\Delta}^{eg}f_*\mathcal{O}_X(D)$ denotes the image of $f_*Tr_{X/Y,\Delta}^{eg}(D)$.
Note that for $e'> e$ the divisor below is effective
$$[(p^{e'g} - 1)\Delta] - p^{(e'-e)g}[(p^{eg} - 1)\Delta].$$
We deduce the following two inclusions
$$S_{\Delta}^{e'g}f_*\mathcal{O}_X(D) \hookrightarrow F_Y^{(e'-e)g*}S_{\Delta}^{eg}f_*\mathcal{O}_X(D)\hookrightarrow F_Y^{e'g*}f_*\mathcal{O}_X(D)$$
by the factorization
\begin{align*}
f_*Tr_{X/Y,\Delta}^{e'g}(D): &f_*\mathcal{O}_X ((1-p^{e'g})K_{X/Y} - [(p^{e'g} - 1)\Delta]+ p^{e'g}D) \\
& \hookrightarrow f_*\mathcal{O}_X ((1-p^{e'g})K_{X/Y} - p^{(e'-e)g}[(p^{eg} - 1)\Delta] + p^{e'g}D) \\
&\xrightarrow{f_*Tr_{X/Y,0}^{(e'-e)g}((1-p^{eg})K_{X/Y} - [(p^{eg} - 1)\Delta]+ p^{eg}D)} \\
&F_Y^{(e'-e)g*}f_*\mathcal{O}_X ((1-p^{eg})K_{X/Y} - [(p^{eg} - 1)\Delta]+ p^{eg}D)\\
 &\xrightarrow{F_Y^{(e'-e)g*}f_*Tr_{X/Y,\Delta}^{eg}(D)} F_Y^{(e'-e)g*}S_{\Delta}^{eg}f_*\mathcal{O}_X(D) \hookrightarrow F_Y^{e'g*}f_*\mathcal{O}_X(D).
\end{align*}

\begin{Proposition}\label{stable-dim-II}
Let the notation be as in Notation \ref{3}. Assume moreover that $D$ is Cartier and $p \nmid \mathrm{ind} ((K_X + \Delta)_{\eta})$. Let $g$ be a positive integer such that $(1-p^{g})(K_X + \Delta)_{\eta}$ is Cartier. Then for every positive integer $e$,
$$\dim_{k(\bar{\eta})}S^{eg}_{\Delta_{\bar{\eta}}}(X_{\bar{\eta}}, D_{\bar{\eta}}) = \mathrm{rank} S_{\Delta}^{eg}f_*\mathcal{O}_X(D).$$
As a consequence, for sufficiently large $e$, $\mathrm{rank} S_{\Delta}^{eg}f_*\mathcal{O}_X(D) = \dim_{k(\bar{\eta})} S^{0}_{\Delta_{\bar{\eta}}}(X_{\bar{\eta}}, D_{\bar{\eta}})$.
\end{Proposition}
\begin{proof}
Shrinking $Y$ and $X$, we can assume that $(1-p^{g})(K_X + \Delta)$ is Cartier. Then we are done by applying Proposition \ref{stable-dim}.
\end{proof}

\subsection{Weak positivity}\label{wp}

\begin{Definition}
A torsion free coherent sheaf $\mathcal{F}$ on a normal quasi-projective variety $Y$ is said to be \emph{generically globally generated} if for a general closed point $y \in Y$ the homomorphism $H^0(Y, \mathcal{F}) \rightarrow \mathcal{F}\otimes k(y)$ is surjective; and is said to be \emph{weakly positive}, if for every ample line bundle $H$ on $Y$ and positive integer $m$, there exists a sufficiently large integer $n$ such that, $S^n(H\otimes S^m(\mathcal{F})^{**})$ is generically globally generated, where for a coherent sheaf $\mathcal{G}$, $\mathcal{G}^{**}:=\mathcal{H}om(\mathcal{H}om(\mathcal{G}, \mathcal{O}_Y), \mathcal{O}_Y)$ denotes the double dual.
\end{Definition}

Recall an invariant introduced by Ejiri in \cite[Sec. 4]{Ej} to measure the positivity of a sheaf.
\begin{Definition}
Let $Y$ be a quasi-projective variety, $\mathcal{F}$ a torsion free coherent sheaf and $H$ an ample $\mathbb{Q}$-Cartier $\mathbb{Q}$-divisor on $Y$. Let
\begin{align*}
t(Y,\mathcal{F}, H)= \sup\{a \in \mathbb{Q}|\mathrm{the~sheaf~}&(F_{Y}^{e*}\mathcal{F})\otimes\mathcal{O}_{Y}([ -p^eaH]) \\
~&\mathrm{is~ generically~globally~generated~ for ~some ~} e>0\}.
\end{align*}
\end{Definition}

Recall the following result due to Ejiri \cite[Proposition 4.7]{Ej}.
\begin{Lemma}\label{cri-for-positivity}
Let $Y$ be a normal quasi-projective variety, $H$ an ample $\mathbb{Q}$-Cartier $\mathbb{Q}$-divisor on $Y$, $Y_0 \subseteq Y$ an open set such that $codim_Y(Y\setminus Y_0) \geq 2$, $\mathcal{F}$ a torsion free coherent sheaf on $Y$.
If $t(Y_0, \mathcal{F}|_{Y_0}, H) \geq 0$, then $\mathcal{F}$ is weakly positive.
\end{Lemma}

\begin{Remark}\label{t-non-ngtv}
The condition $t(Y, \mathcal{F}, H) \geq 0$ is equivalent to that, there exist a sequence of positive integers $\{n_e|e = 1,2,3,\cdots \}$ such that $\frac{n_e}{p^e}\rightarrow 0$, $n_eH$ is Cartier, and the sheaf
$(F_{Y}^{e*}\mathcal{F})\otimes\mathcal{O}_{Y}(n_eH)$ is generically globally generated.

It is easy to show that if $t(Y, \mathcal{F}, H) \geq 0$ then $t(Y, \mathcal{F}, H') \geq 0$ for every ample divisor $H'$. So if this happens we may simply denote $t(Y, \mathcal{F}) \geq 0$.
\end{Remark}

\subsection{Surjection of restriction maps}
Recall a Keeler's result.
\begin{Lemma}[{\cite[Theorem 1.5]{Keeler}}](\textbf{\emph{Relative Fujita Vanishing}})
Let $f: X \rightarrow Y$ be a projective morphism over a Noetherian scheme, $H$ an $f$-ample line bundle and $\mathcal{F}$ a coherent sheaf on $X$. Then there exists a positive integer $N$ such that, for every $n >N$ and every nef line bundle $L$
$$R^if_*(\mathcal{F}\otimes H^n \otimes L) = 0, \mathrm{~if~} i>0.$$
\end{Lemma}

\begin{Lemma}\label{restr}
Let $f: X \rightarrow Y$ be a surjective projective morphism between two projective varieties. Let $H$ be an ample line bundle and $\mathcal{F}$ a coherent sheaf on $X$. Then there exist a positive integer $N$ and a non-empty Zariski open set $Y_0 \subseteq Y$ such that, for every $n>N$, every nef line bundle $L$ on $X$ and every closed point $y\in Y_0$ the restriction map
$$r_y^{n, L}: H^0(X, \mathcal{F} \otimes H^n \otimes L) \rightarrow H^0(X_y, \mathcal{F} \otimes H^n \otimes L\otimes \mathcal{O}_{X_y})$$
is surjective.

In particular for two nef Cartier divisors $A_1 ,A_2$ on $X$, if $A_1 + A_2$ is ample, then there exists an integer $M$ such that, for integers $m, n>M$ and a closed point $y\in Y_0$ the restriction map below is surjective
$$r_y^{m,n}: H^0(X, \mathcal{F} \otimes \mathcal{O}_X(mA_1 + nA_2)) \rightarrow H^0(X_y, \mathcal{F} \otimes \mathcal{O}_X(mA_1 + nA_2) \otimes \mathcal{O}_{X_y}).$$
\end{Lemma}
\begin{proof}
Consider the following commutative diagram
\begin{center}\xymatrix@C=2.5cm{&X_{\Delta} \ar[d]^{f_{\Delta}}\ar[r]^{j} &X \times Y \ar[d]^{f \times id_Y}\ar[r]^{p_1} &X \ar[d]^f \\
                    &\Delta \ar@{^(->}[r]^{i}          &Y \times Y  \ar[r]^{q_1}            &Y
}
\end{center}
where $i:\Delta \hookrightarrow Y \times Y$ denotes the diagonal embedding of $Y$, $X_{\Delta} = (X \times Y) \times_{Y\times Y}\Delta$, $p_i, q_i, i = 1,2$ denote the projection from $X \times Y, Y\times Y$ to the $i^{\mathrm{th}}$ factors respectively.

Denote by $\mathcal{K}$ the kernel of the restriction homomorphism $p_1^* \mathcal{F}  \rightarrow p_1^*\mathcal{F}\otimes \mathcal{O}_{X_{\Delta}}$.
Applying relative Fujita vanishing above, since $p_1^*H$ is $p_2$-ample, we can find a positive integer $N$ such that, for every $n>N$, $i>0$ and every nef line bundle $L$ on $X$
$$R^ip_{2*} (\mathcal{K} \otimes p_1^*(H^n \otimes L)) = 0~\mathrm{and}~R^if_*(\mathcal{F} \otimes H^n \otimes L) = 0~~~~~~~(\clubsuit).$$

Let $n>N$. Tensoring the exact sequence
$$0 \rightarrow \mathcal{K} \rightarrow p_1^*\mathcal{F} \rightarrow p_1^*\mathcal{F} \otimes \mathcal{O}_{X_{\Delta}} \rightarrow 0$$
by the line bundle $p_1^*(H^n \otimes L)$ yields the exact sequence
$$0 \rightarrow \mathcal{K} \otimes p_1^*(H^n \otimes L) \rightarrow p_1^*(\mathcal{F}\otimes H^n \otimes L) \rightarrow
p_1^*(\mathcal{F}\otimes H^n \otimes L) \otimes\mathcal{O}_{X_{\Delta}} \rightarrow 0.$$
Applying the derived functor $Rp_{2*}$ to the exact sequence above, by vanishing $\clubsuit$ we get a surjection
$$\alpha^{n, L}: p_{2*} p_1^*(\mathcal{F}\otimes H^n \otimes L) \twoheadrightarrow p_{2*} (p_1^*(\mathcal{F}\otimes H^n \otimes L) \otimes \mathcal{O}_{X_{\Delta}}).$$
Identifying $X_{\Delta}$ with $X$ via the isomorphism $p_1 \circ j$, and $p_2|_{X_{\Delta}}: X_{\Delta} \rightarrow Y$ with $f$, we can identify $p_{2*} (p_1^*(\mathcal{F}\otimes H^n \otimes L) \otimes\mathcal{O}_{X_{\Delta}})$ with $f_*(\mathcal{F}\otimes H^n \otimes L)$.
Then by
$$p_{2*} p_1^*(\mathcal{F}\otimes H^n \otimes L) \cong H^0(X, \mathcal{F}\otimes H^n \otimes L) \otimes \mathcal{O}_Y \cong H^0(Y, f_*(\mathcal{F}\otimes H^n \otimes L)) \otimes \mathcal{O}_Y,$$
we can identify the map $\alpha^{n,L}$ with the natural map
$$\beta^{n, L}: H^0(Y, f_*(\mathcal{F}\otimes H^n \otimes L)) \otimes \mathcal{O}_Y \twoheadrightarrow f_*(\mathcal{F}\otimes H^n \otimes L).$$

There exists a non-empty open set $Y_0 \subseteq Y$ such that $\mathcal{F}$ is flat over $Y_0$. Since $R^1f_*(\mathcal{F} \otimes H^n \otimes L) = 0$ by $\clubsuit$, applying \cite[Theorem 12.11]{Har} we have that
for every closed point $y \in Y_0$,
$$f_*(\mathcal{F}\otimes H^n \otimes L) \otimes k(y)\cong H^0(X_y, \mathcal{F}\otimes H^n \otimes L \otimes \mathcal{O}_{X_y}).$$
Since $\beta_{n, L}$ is a surjection, we conclude that the restriction map
$$r_y^{n, L}: H^0(X, \mathcal{F} \otimes H^n \otimes L) \rightarrow H^0(X_y, \mathcal{F} \otimes H^n \otimes L\otimes \mathcal{O}_{X_y}).$$
is a surjection.

The remaining assertion is an easy consequence of the first one.
\end{proof}

\subsection{Minimal models of $3$-folds}
We collect some results on minimal model theory for $3$-folds, which will be used in this paper.

First recall a result of Kawamata adapted to char $p>0$, and please refer to [\cite{BW}, Lemma 5.6] for a proof.
\begin{Lemma}\label{l-linear-pullback}
Let $f: X \rightarrow Z$ be a fibration between normal quasi-projective varieties over an algebraically closed field $k$ with $\mathrm{char}~k = p >0$,
Let $L$ be a nef $\mathbb{Q}$-Cartier $\mathbb{Q}$-divisor on $X$ such that $L|_F\sim_{\mathbb{Q}} 0$ where $F$ is the generic
fibre of $f$. Assume $\dim Z\le 3$. Then there exist a commutative diagram
$$
\xymatrix{
X'\ar[r]^\phi\ar[d]^{f'} & X\ar[d]^f\\
Z'\ar[r]^\psi & Z
}
$$
with $\phi,\psi$ projective birational, and an $\mathbb{R}$-Cartier divisor $D$ on $Z'$ such that
$\phi^* L\sim_{\mathbb{Q}} f'^*D$.
\end{Lemma}

\begin{Theorem}\label{rel-mmp}
Let $(X,\Delta)$ be a projective $\mathbb{Q}$-factorial klt pair of dimension 3 and $f: X\rightarrow  Y$ a fibration over an algebraically closed field $k$ with $\mathrm{char}~k = p >5$.

(1) If $K_X+\Delta$ is pseudo-effective over $Y$, then $(X,\Delta)$ has a log minimal model over $Y$.

(2) If $K_X+\Delta$ is not pseudo-effective over $Y$, then $(X,\Delta)$ has a Mori fibre space over $Y$.

(3) Assume that $K_X+\Delta$ is nef over $Y$.
\begin{itemize}
\item[(3.1)]
If $K_X+\Delta$ or $\Delta$ is big over $Y$, then $K_X+\Delta$ is semi-ample over $Y$.
\item[(3.2)]
If $Y$ is a smooth curve and $\kappa(X_{\eta}, (K_X+\Delta)_{\eta}) \geq 0$, then $(K_X+\Delta)_{\eta}$ is semi-ample on $X_{\eta}$.
\item[(3.3)]
If $Y$ is a smooth curve and $\kappa(X_{\eta}, (K_X+\Delta)_{\eta}) = 0$ or $2$, then $K_X+\Delta$ is semi-ample over $Y$.
\item[(3.4)]
If $Y$ is a smooth curve with $g(Y) \geq 1$ and $\kappa(X_{\eta}, (K_X+\Delta)_{\eta}) \geq 0$, then $K_X+\Delta$ is nef.
\end{itemize}

(4) If $Y$ is a non-uniruled surface and $K_X+\Delta$ is pseudo-effective over $Y$, then $K_X + \Delta$ is pseudo-effective, and there exists a map to a minimal model $\sigma: X \dashrightarrow \bar{X}$ such that, the restriction $\sigma|_{X_{\eta}}$ is an isomorphism from $X_{\eta}$ to its image.
\end{Theorem}
\begin{proof}
For (1) refer to \cite{HX} and \cite{Bir13}.

For (2) refer to \cite{BW}.

For (3.1)refer to \cite{Bir13}, \cite{Xu} and \cite{BW}.

For (3.2) and (3.3) refer to \cite[Theorem 1.5 and 1.6 and the remark below 1.6]{BCZ} or \cite[Theorem 1.1]{Ta15}.

Assertion (3.4) follows from the cone theorem \cite[Theorem 1.1]{BW}. Indeed, otherwise we can find an extremal ray $R$ generated by a rational curve $\Gamma$,  so $\Gamma$ is contained in a fiber of $f$ since $g(Y) >0$, this contradicts that $K_X + \Delta$ is $f$-nef.

For (4), first $K_X + \Delta$ is obviously pseudo-effective because otherwise, $X$ will be ruled by horizontal rational curves (w.r.t. $f$) by (2), which contradicts that $Y$ is non-uniruled. The exceptional locus of a flip contraction is of dimension one, so it does not intersect $X_{\eta}$; neither does that of an extremal divisorial contraction, because it is uniruled (cf. the $2^{\mathrm{nd}}$ paragraph of the proof of \cite[Lemma 3.2]{BW}), hence does not dominate over $Y$. Running an LMMP for $K_X +\Delta$, by induction we get a map $\sigma: X \dashrightarrow \bar{X}$ as required.
\end{proof}

\subsection{Covering Theorem}
The result below is \cite[Theorem 10.5]{Iit} when $X$ and $Y$ are both smooth, and the proof also applies when they are normal.
\begin{Theorem}\label{cth}
Let $f: X \rightarrow Y$ be a proper surjective morphism between normal projective varieties. If $D$ is a Cartier divisor on $Y$ and $E$ an effective $f$-exceptional Cartier divisor on $X$. Then
$$\kappa(X,f^*D + E) = \kappa(Y,D).$$
\end{Theorem}

\subsection{Easy subadditivity of Kodaira dimensions}
The following result is known to experts, please refer to \cite[Lemma 2.22]{BCZ} or \cite[Lemma 4.2]{Pa2} for a proof.
\begin{Lemma}\label{l-adtv-of-kdim}
Let $f: X\rightarrow Y$ be a fibration between normal projective varieties.
Let $D$ be an effective $\mathbb{Q}$-Cartier divisor on $X$ and $H$ a big $\mathbb{Q}$-Cartier
divisor on $Y$. Then
$$\kappa(D + f^*H) \geq \kappa(X_{\eta}, D|_{X_{\eta}}) + \dim Y.$$
\end{Lemma}

\section{Proof of Theorem \ref{mthp}}\label{pf-positivity}
Let $Y_0$ be a smooth open subset of $Y$ such that $\mathrm{codim}_Y(Y\setminus Y_0) \geq 2$.

By Proposition \ref{stable-dim-II} we can assume $g$ is divisible enough that for every positive integer $e$, the sheaf $S_{\Delta}^{eg}f_*\mathcal{O}_X(D)$ has the stable rank
$\dim_{k(\bar{\eta})} S_{\Delta_{\bar{\eta}}}^0(X_{\bar{\eta}}, D_{\bar{\eta}})$. Then for every integer $e >0$, the composite homomorphism below is generically surjective
\begin{align*}
\alpha^{eg}: f_*\mathcal{O}_X ((1-p^{eg})(K_{X/Y})- &[(p^{eg}-1)\Delta] + p^{eg}D)|_{Y_0} \\
         &\twoheadrightarrow (S_{\Delta}^{eg}f_*\mathcal{O}_X(D))|_{Y_0} \hookrightarrow (F_Y^{(e-1)g*}S_{\Delta}^{g}f_*\mathcal{O}_X(D))|_{Y_0},
\end{align*}
because the two sheaves $S_{\Delta}^{eg}f_*\mathcal{O}_X(D)|_{Y_0}$ and $(F_Y^{(e-1)g*}S_{\Delta}^{g}f_*\mathcal{O}_X(D))|_{Y_0}$ have the same rank.

Let $H$ be an ample Cartier divisor on $Y$. Tensoring the map $\alpha^{eg}$ with $\mathcal{O}_Y(eH)$, we get a generically surjective homomorphism
\begin{align*}
\beta^{eg}: f_*\mathcal{O}_X ((1-p^{eg})(K_{X/Y})- &[(p^{eg}-1)\Delta] + p^{eg}D+ ef^*H)|_{Y_0} \\
&\rightarrow F_Y^{(e-1)g*}S_{\Delta}^{g}f_*\mathcal{O}_X(D)\otimes \mathcal{O}_Y(eH)|_{Y_0}.
\end{align*}

\begin{Claim}
There is an integer $e_0$ such that, for every integer $e > e_0$ the sheaf $f_*\mathcal{O}_X ((1-p^{eg})K_{X/Y}- [(p^{eg}-1)\Delta] + p^{eg}D + ef^*H)$ is generically globally generated.
\end{Claim}
\begin{proof}[Proof of the claim]
Since $D -K_{X/Y} - \Delta$ is $f$-semi-ample, we have two morphisms $h: X \rightarrow Z$ and $g: Z \rightarrow Y$, such that $D -K_{X/Y} - \Delta \sim_{\mathbb{Q}} h^*A'$ where $A'$ is a $g$-ample $\mathbb{Q}$-Cartier divisor on $Z$, which is also nef by the assumption. Take an integer $d >0$ such that $A = dA' \sim d(D -K_{X/Y} - \Delta)$ is Cartier. Write that $p^{eg} -1 = q_e d + r_e $ where $q_e$ and $r_e$ are integers such that $0 \leq r_e <d$. Then by $f_* = g_*\circ h_*$, we have
\begin{align*}
&f_*\mathcal{O}_X ((1-p^{eg})K_{X/Y}- [(p^{eg}-1)\Delta]+ p^{eg}D + ef^*H)\\
&\cong f_*\mathcal{O}_X (q_e d (D - K_{X/Y} - \Delta) + ef^*H  + (r_e+1)D - r_eK_{X/Y} - [r_e\Delta])\\
& \cong g_* h_* \mathcal{O}_X (h^*q_eA + eh^*g^*H  + (r_e+1)D - r_eK_{X/Y} - [r_e\Delta])\\
&\cong g_*(\mathcal{O}_Z(q_eA + eg^*H)\otimes h_*\mathcal{O}_X ((r_e+1)D - r_eK_{X/Y} - [ r_e\Delta]))
\end{align*}
where the last ``$\cong$'' is from using projection formula. Note that the set
$$\{h_*\mathcal{O}_X ((r_e+1)D - r_eK_{X/Y} - [r_e\Delta])| e = 0,1,2,\cdots \}$$
contains finitely many coherent sheaves.
Since $A + g^*H$ is ample, and both $A$ and $g^*H$ are nef, by Lemma \ref{restr} there exist a positive integer $e_0$ and a non-empty Zariski open subset $Y'_0 \subseteq Y$ such that for every $e>e_0$ and $y \in Y'_0$, the restriction map
\begin{align*}
&H^0(Y, f_*\mathcal{O}_X ((1-p^{eg})(K_{X/Y})- [(p^{eg}-1)\Delta]+ p^{eg}D + ef^*H)) \\
&\cong H^0(Z, \mathcal{O}_Z(q_eA + eg^*H)\otimes h_*\mathcal{O}_X ((r_e+1)D - r_eK_{X/Y} - [ r_e\Delta])) \\
&\xrightarrow{}  H^0(Z_y, \mathcal{O}_Z(q_eA + eg^*H)\otimes h_*\mathcal{O}_X ((r_e+1)D - r_eK_{X/Y} - [ r_e\Delta])
\otimes\mathcal{O}_{Z_y}) \\
&\cong g_*(\mathcal{O}_Z(q_eA + eg^*H)\otimes h_*\mathcal{O}_X ((r_e+1)D - r_eK_{X/Y} - [ r_e\Delta]))\otimes k(y) \\
& \cong f_*\mathcal{O}_X ((1-p^{eg})(K_{X/Y})- [(p^{eg}-1)\Delta] + p^{eg}D + ef^*H)\otimes k(y)
\end{align*}
is surjective.
\end{proof}
The claim above implies that the image of $\beta^{eg}$ is generically globally generated, hence so is the sheaf $F_Y^{(e-1)g*}S_{\Delta}^{g}f_*\mathcal{O}_X(D)\otimes \mathcal{O}_Y(eH)|_{Y_0}$. Therefore, by Remark \ref{t-non-ngtv} we have $t(Y_0, S_{\Delta}^{g}f_*\mathcal{O}_X(D)|_{Y_0}, H) \geq 0$, which implies that $S_{\Delta}^{g}f_*\mathcal{O}_X(D)$ is weakly positive by Lemma \ref{cri-for-positivity}.

If $Y$ is smooth, then setting $Y_0 = Y$, by the argument above we show that $t(Y, S_{\Delta}^{g}f_*\mathcal{O}_X(D), H) \geq 0$.

\section{Subadditivity of Kodaira dimensions}\label{pf-of-subadd}
In this section, we will prove Theorem \ref{mthk}. Let's begin with a theorem with similar spirit of \cite[Lemma 4.4]{Pa2}.
\begin{Theorem}\label{F-p-subadd-of-kod-dim}
Let $f: X \rightarrow Y$ be a separable fibration between normal projective varieties over an algebraically closed field $k$ with $\mathrm{char}~k = p>0$, and let $D$ be a Cartier divisor on $X$. Assume that for some positive integer $e$,  the sheaf $F_Y^{e*} f_*\mathcal{O}_X(D)$ contains a non-zero subsheaf $\mathcal{F}$ such that $t(Y, \mathcal{F}) \geq 0$. Then for any big $\mathbb{Q}$-Cartier divisor $H$ on $Y$, we have
$$\kappa(D + H) \geq \kappa(X_{\eta}, D|_{X_{\eta}}) + \dim Y.$$
\end{Theorem}
\begin{proof}
Let $A$ be an ample $\mathbb{Q}$-Cartier divisor on $Y$ such that $H \geq 2A$.
By Remark \ref{t-non-ngtv}, we can find positive integers $g, n_g \ll p^g$ such that, the sheaf $F_{Y}^{g*}\mathcal{F} \otimes \mathcal{O}_Y(n_gA)$ is generically globally generated.

Consider the following commutative diagram \\
\[\xymatrix@C=2cm{&X'\ar@/^1.5pc/[rr]|{\sigma}\ar[r]^{\sigma'}\ar[dr]^{f'} &X_{Y^{e+g}}\ar[r]^{\pi_{Y}^{e+g}}\ar[d]^{f_{e+g}}     &X\ar[d]^f\\
&      &Y^{e+g}\ar[r]^{F_{Y}^{e+ g}}     &Y\\
} \]
where $X'$ denotes the normalization of $X_{Y^{e+g}}$ and $\sigma, \sigma', f'$ denote the natural morphisms.

By the commutative diagram above, there are natural inclusions
$$F_{Y}^{(e+ g)*}f_*\mathcal{O}_X(D) \hookrightarrow f_{(e+g)*}\pi_{Y}^{e+g*} \mathcal{O}_X(D) \hookrightarrow f'_*\mathcal{O}_{X'}(\sigma'^* \pi_{Y}^{(e+g)*}D) = f'_*\mathcal{O}_{X'}(\sigma^*D).$$
Therefore, the sheaf $f'_*\mathcal{O}_{X'}(\sigma^*D)\otimes \mathcal{O}_Y(n_gA)$ contains a generically globally generated subsheaf $F_{Y}^{g*}\mathcal{F} \otimes \mathcal{O}_Y(n_gA)$. We can find an effective divisor $D'$ on $X'$ such that
$$D' \sim \sigma^*D + n_gf'^*A.$$
By $F_{Y}^{(e+g)*}H = p^{e+g}H > 2n_g A$, we complete the proof by
\begin{align*}
\kappa(X, D+ f^*H)  &=  \kappa(X', \sigma^*D + f'^*F_{Y}^{(e+g)*}H) \cdots \text{by Theorem \ref{cth}} \\
&\geq \kappa(X', \sigma^*D + 2n_g f'^*A) \\
  & =  \kappa(X', D' + n_g f'^*A) \\
  & \geq \kappa(X'_{\eta^{e+g}}, D'|_{X'_{\eta^{e+g}}}) + \dim Y  \cdots \text{by Lemma \ref{l-adtv-of-kdim}}\\
  & \geq \kappa(X_{\eta}, D|_{X_{\eta}}) + \dim Y \cdots \text{since $D'|_{X'_{\eta^{e+g}}} = \sigma^*(D|_{X_{\eta}})$}.
\end{align*}

\end{proof}

Before proving Theorem \ref{mthk}, let's explain the strategy. If the base $Y$ is smooth and $A$ is Cartier (good situation), first by Theorem \ref{mthp} we can show that some Frobenius pullback of $f_*\mathcal{O}_X(D-f^*A)$ contains a nonzero subsheaf $\mathcal{F}$ with $t(Y, \mathcal{F}) \geq 0$, then by $D = (D- f^*A) + f^*A$, applying Theorem \ref{F-p-subadd-of-kod-dim} we finish the proof. To reduce to a fibration with the base being smooth, we will do a smooth alteration base change (\cite{J}), namely, a proper, surjective and generically finite morphism; and to reduce to the situation $A$ being Cartier, we will do some Frobenius base changes and replace the pullback of $A$ with a big Cartier divisor.

\begin{proof}[Proof of Theorem \ref{mthk}] We break the proof into three steps following the above strategy.

\textbf{Step 1:} We reduce to the good situation.

Consider the following commutative diagram
\begin{equation*}
\begin{gathered}
\xymatrix@C=1.5cm{&X_2\ar@/^1.6pc/[rrr]|{\sigma}\ar[rd]^{f_2}\ar[r]^>>>>>>>{\sigma_2} &\bar{X}_2=X\times_Y Y_2\ar[rr]^{\sigma_1}\ar[d]^{\bar{f}_2} & &X\ar[d]^{f}\\
&     &Y_2 = Y_1^{g_0}\ar[r]^{F_{Y_1}^{g_0}} &Y_1\ar[r]^{\mu} &Y\\
}
\end{gathered}
\end{equation*}
where
\begin{itemize}
\item
if $Y$ is smooth then we set $Y_1 = Y$ and $\mu = \mathrm{id}_Y$; otherwise, $f: X \to Y$ is flat by the assumption, we set $\mu: Y_1\rightarrow Y$ to be a smooth alteration (\cite{J}), thus by the construction the base change $\bar{X}_2=X\times_Y Y_2$ is always integral;
\item
$\sigma_2: X_2 \rightarrow \bar{X}_2$ is the normalization morphism;
\item
$\sigma, \sigma_1, f_2$ and $\bar{f}_2$ denote the natural morphisms.
\end{itemize}
We can assume $g_0$ is big enough such that, the geometric generic fiber $(X_2)_{\bar{\eta}}$ is normal, and that the integral part $A_2 = [p^{g_0}\mu^*A]$ is big. Then $B = p^{g_0}\mu^*A - A_2$ is effective, and on $Y_2$ we have $F_{Y_1}^{g_0*}\mu^*A = A_2 + B$.

\medskip

We claim that there exist an effective $\mathbb{Q}$-Weil divisor $\Delta'$ and an effective $\sigma$-exceptional Cartier divisor $E_2$ on $X_2$ such that
$$K_{X_2/Y_2} + \Delta' = \sigma^*(K_{X/Y} + \Delta) + E_2~\mathrm{and}~E_2|_{(X_2)_{\bar{\eta}}} = 0.$$
Indeed, if $Y$ is smooth then the base change $Y_2 \to Y$ is flat and thus $K_{\bar{X}_2/Y_2} = \sigma_1^*K_{X/Y}$, we can set $E_2 = 0$ and construct $\Delta'$ by applying Proposition \ref{F-non-vanishing} (4); otherwise, since $f: X\to Y$ is flat, we can apply Proposition \ref{compds} to get the divisors $\Delta'$ and $E_2$ on $X_2$ as required.

\medskip

Let
$$\Delta_2 = \Delta' + f_2^*B ~\mathrm{and}~ D_2 = \sigma^*D + E_2.$$
Immediately it follows that $K_{X_2/Y_2} + \Delta_2 = \sigma^*(K_{X/Y} + \Delta) + E_2 + f_2^*B$ and
$$(K_{X_2/Y_2} + \Delta_2)_{\bar{\eta}} = (\sigma^*(K_{X/Y} + \Delta))_{\bar{\eta}}~\mathrm{and}~(D_2)_{\bar{\eta}} = (\sigma^*D)_{\bar{\eta}}.$$
Therefore,

(i) the divisor $(D_2 - f_2^*A_2) - K_{X_2/Y_2} - \Delta_2  \sim_{\mathbb{Q}} \sigma^*(D- (K_{X/Y} + \Delta) - f^*A)$
is nef and $f_2$-semi-ample;

(ii) $p\nmid \mathrm{ind}(K_{X_2/Y_2} + \Delta_2)_{\bar{\eta}}$ by the assumption (1) in the theorem; and

(iii) applying Proposition \ref{F-non-vanishing} (4) shows $S^0_{(\Delta_2)_{\bar{\eta}}}((X_2)_{\bar{\eta}}, (D_2)_{\bar{\eta}}) \neq 0$ by the assumption (3).

\medskip

\textbf{Step 2:}
Applying Theorem \ref{mthp} on the pair $(X_2, \Delta_2)$, wee can show that for sufficiently divisible $e$, the sheaf $F_{Y_2}^{e*}f_{2*}\mathcal{O}_{X_2}(D_2- f_2^*A_2)$ contains a nonzero subsheaf
$S^e_{\Delta_2}f_{2*}\mathcal{O}_{X_2}(D_2- f_2^*A_2)$ satisfying $t(Y, S^e_{\Delta_2}f_{2*}\mathcal{O}_{X_2}(D_2- f_2^*A_2)) \geq 0$.
Then we conclude that
\begin{align*}
\kappa(X, D) &  =  \kappa(X_2, D_2 = \sigma^*D + E) ~~~~\text{$\cdot\cdot\cdot$ by Theorem \ref{cth}}\\
&   =     \kappa(X_2, (D_2 - f_2^*A_2) + f_2^*A_2 ) \\
&   \geq  \dim Y + \kappa((X_2)_{\eta_2}, (D_2)_{\eta_2}) ~~~~\text{$\cdot\cdot\cdot$ by Theorem \ref{F-p-subadd-of-kod-dim}}\\
&  \geq  \dim Y + \kappa(X_{\bar{\eta}}, D_{\bar{\eta}}) ~~~~\text{$\cdot\cdot\cdot$ since $(D_2)_{\bar{\eta}} = (\sigma^*D)_{\bar{\eta}}$}\\
\end{align*}

\medskip

\textbf{Step 3:} We are left to prove that $D$ is big under the conditions (1), (2') and that $D$ is nef and $f$-big.

Take an ample divisor $A_1$ on $Y$. Then $D+ f^*A_1$ is big. We can write that
$$D+ f^*A_1 \sim_{\mathbb{Q}} H_1 + B_1$$
where $H_1$ is ample and $B_1$ is an effective $\mathbb{Q}$-Cartier divisor with $p \nshortmid \mathrm{ind}(B_1)$.
Take a rational number $\delta > 0$ small enough such that

(i) $A'= A - \delta A_1$ is big on $Y$; and

(ii)  $p \nshortmid \mathrm{ind}(\delta B_1)$.\\
Let $\Delta' = \Delta + \delta B_1$. Then for sufficiently divisible integer $m >0$, since $D$ is nef and $f$-big we have

(a) $mD - (K_{X/Y} + \Delta') - f^*A' = (m - 1- \delta)D + \delta(D + f^*A_1 - B_1) + (D - (K_{X/Y} + \Delta) - f^*A)$ is ample by the condition (2'); and

(b) $S^0_{\Delta'_{\bar{\eta}}}(X_{\bar{\eta}}, mD_{\bar{\eta}}) \neq 0$ by Proposition \ref{F-non-vanishing} (3).

Finally applying the result in \textbf{Step 2} on the pair $(X, \Delta')$, we show that $mD$ is big.
\end{proof}

\section{Application to three-folds}\label{app-to-3fold}
In this section we will focus on three dimensional varieties in characteristic $p>5$. Taking advantages of minimal model program and smooth resolution of singularities, both Corollary \ref{app-to-3dim} and \ref{app-to-3dim-special} follow easily from Theorem \ref{mthk}.

\subsection{Proof of Corollary  \ref{app-to-3dim}}
We first pass to a fibration over $Z$. Take a smooth resolution $\sigma: W \to X$, and assume the morphism $W \to Y$ lifts to a fibration $g: W \to Z$, which fit into the following commutative diagram
$$\xymatrix{&W\ar[d]^{g}\ar[r]^{\sigma} &X \ar[d]^{f}\\
& Z \ar[r]^{\mu}    &Y\\
}$$
Then $\sigma^*(K_X + \Delta)$ is nef and $g$-big, thus $n\sigma^*(K_X + \Delta) + K_W$ is $g$-big for sufficiently big $n$, and
$$(n\sigma^*(K_X + \Delta) + K_W) - K_{W/Z} - f^*K_Z = n\sigma^*(K_X + \Delta)$$
is nef.
By Proposition \ref{F-non-vanishing} (3), for sufficiently divisible $n>0$
$$S^{0}(W_{\bar{\eta}}, (n\sigma^*(K_X + \Delta) + K_W)_{\bar{\eta}}) \neq 0.$$
Since $Z$ is smooth and $K_Z$ is big, applying Theorem \ref{mthk} shows that $n\sigma^*(K_X + \Delta) + K_W$ is big.

There exist an effective divisor $D$ and an effective $\sigma$-exceptional divisor $E$ on $W$ such that $K_W \sim_{\mathbb{Q}} \sigma^*(K_X + \Delta) - D + E$. Applying Theorem \ref{cth}, we can show
$$\kappa(X, K_X + \Delta) = \kappa(W, \sigma^*(n+1)(K_X + \Delta) + E) \geq \kappa(W, \sigma^*n(K_X + \Delta) + K_W) = 3.$$

\subsection{Proof of Corollary  \ref{app-to-3dim-special}}\label{pf-3dim}
Before giving the proof, we remark some easy results. Let $\rho: X' \to X$ be a smooth log resolution of $(X,\Delta)$. We can write that
$\small{K_{X'} + \rho_*^{-1}\Delta + \sum_ia_iE_i = \rho^*(K_X + \Delta) + \sum_j b_jF_j}$
where $E_i, F_j$ are distinct reduced and irreducible exceptional components and $0< a_i <1, b_j \geq 0$. Let $\Delta' = \rho_*^{-1}\Delta + \sum_ia_iE_i$. Then $(X', \Delta')$ is klt, and by Theorem \ref{cth} we conclude
$$\kappa(X', K_{X'} + \Delta') = \kappa(X, K_X + \Delta)~\mathrm{and}~\kappa(X', K_{X'}) \leq \kappa(X, K_X).$$
Moreover we have $\kappa(X'_{\bar{\eta}}, (K_{X'} + \Delta')_{\bar{\eta}}) \geq \kappa(X_{\bar{\eta}}, (K_X + \Delta)_{\bar{\eta}})$, and
in case (2) $\kappa(X'_{\bar{\eta}}, K_{X'_{\bar{\eta}}}) =\kappa(X_{\bar{\eta}}, K_{X_{\bar{\eta}}})$ since $X_{\bar{\eta}}$ is assumed smooth.
So to prove the inequality of this corollary, we are allowed to replace $(X,\Delta)$ with $(X', \Delta')$ in case (1) and replace $X$ with $X'$ in case (2).

Let $\sigma: (X, \Delta) \dashrightarrow(\bar{X}, \bar{\Delta})$ be the map to a minimal model of $(X, \Delta)$. If necessary, by replacing $(X, \Delta)$ with a smooth log resolution as above, we can assume $\sigma$ is a morphism. Since $Y$ is non-uniruled, applying Theorem \ref{rel-mmp} (3.4) and (4), we have the following commutative diagram
$$\xymatrix{&(X,\Delta)\ar[d]^{f}\ar[r]^{\sigma} &(\bar{X}, \bar{\Delta}) \ar@{-->}[dl]^{\bar{f}}\\
& Y     &\\
}$$
here if $\dim Y =1$ then $\bar{f}: \bar{X} \dashrightarrow Y$ is a morphism, and if $\dim Y =2$ then there exists a nonempty open subset $U \subseteq Y$ such that $\bar{X}_U \cong X_U$, thus the restriction map $\bar{f}: \bar{X}_U \to U$ is a morphism.

In case (1), by the construction, applying Theorem \ref{rel-mmp} (4) shows that $\sigma^*(K_{\bar{X}} + \bar{\Delta})$ is nef and $f$-big. For sufficiently divisible $n>0$ the divisor
$$(n\sigma^*(K_{\bar{X}} + \bar{\Delta}) + K_X) - K_{X/Y} - f^*K_Y = n\sigma^*(K_X + \Delta)$$
is nef and $f$-big, and by Proposition \ref{F-non-vanishing} (3)
$$S^{0}(X_{\bar{\eta}}, (n\sigma^*(K_{\bar{X}} + \bar{\Delta}) + K_X)_{\bar{\eta}}) \neq 0.$$
Since $K_Y$ is big, applying Theorem \ref{mthk} shows that $n\sigma^*(K_{\bar{X}} + \bar{\Delta}) + K_X$ is big.
Then arguing as in the last paragraph of the proof of Corollary \ref{app-to-3dim}, we can show that $K_X + \Delta$ is big.

In case (2), it is assumed $\Delta = 0$. Then since $WC_{3,2}$ has been proved in \cite{CZ}, we can assume $Y$ is a curve with $g(Y) >1$, and have a fibration $\bar{f}: \bar{X} \rightarrow Y$. The case $\kappa(X_{\bar{\eta}}) = 2$ is included in Case (1).
We only need to consider the cases $\kappa(X_{\bar{\eta}}) = 0$ or $1$.

If $\kappa(X_{\bar{\eta}}) = 0$, then $K_{\bar{X}/Y}$ is  relatively semi-ample over $Y$ by Theorem \ref{rel-mmp} (3.3). Notice that general fibers of $\bar{f}$ have canonical singularities, which are strongly $F$-regular by \cite{Har98} because $\mathrm{char}~k >5$. Applying \cite[Theorem 3.16]{Pa}, we have that $K_{\bar{X}/Y}$ is nef, so there exists a nef $\mathbb{Q}$-Cartier $\mathbb{Q}$-divisor $M$ on $Y$ such that
$$K_{\bar{X}/Y} \sim_{\mathbb{Q}} \bar{f}^*M.$$
It is easy to conclude that
$$\kappa(X) = \kappa(\bar{X}, K_{\bar{X}}) = \kappa(\bar{X}, K_{\bar{X}/Y} + \bar{f}^*K_Y) = \kappa(Y, K_Y + M) = 1 = \dim Y.$$

If $\kappa(X_{\bar{\eta}}) = 1$, then $X_{\bar{\eta}}$ is a smooth surface over $k(\bar{\eta})$, and general fibers of its Iitaka fibration are smooth elliptic curves.
Considering the relative Iitaka fibration of $X$, if necessary, by blowing up $X$, we can assume $f:X\to Y$ factors through an elliptic fibration $h: X\to Z$ to a smooth surface $Z$, which fit into the following commutative diagram
$$\xymatrix{&X \ar[r]^{\sigma}\ar[d]^h &\bar{X} \ar[d]^{\bar{f}}\\
&Z\ar[r]^{g}   &Y\\
}$$
Applying Lemma \ref{l-linear-pullback}, if necessary, again by blowing up both $X$ and $Z$, we can also assume $\sigma^*K_{\bar{X}} \sim_{\mathbb{Q}} h^*H$ for some nef and $g$-big $\mathbb{Q}$-Cartier $\mathbb{Q}$-divisor $H$ on $Z$.
By Proposition \ref{F-non-vanishing} (3), for sufficiently divisible $n >0$, $S^0(Z_{\bar{\eta}}, (nH +K_Z)_{\bar{\eta}}) \neq 0$. Combining that $(nH +K_Z) - K_{Z/Y} - g^*K_Y$ is nef and $g$-big and that $K_Y$ is big, by Theorem \ref{mthk} we can show $nH +K_Z$ is big.
Since $\kappa(X, K_{X/Z}) \geq 0$ (\cite[3.2]{CZ}), there exists an effective $\mathbb{Q}$-divisor $E$ on $X$ such that $K_{X/Z} \sim_{\mathbb{Q}} E$. Applying Theorem \ref{cth} it follows that
\begin{align*}
\kappa(X, K_X) &= \kappa(X, (n+1)K_X) =  \kappa(X, n\sigma^*K_{\bar{X}} + K_X)\\
               &= \kappa(X, n\sigma^*K_{\bar{X}} + E + h^*K_Z) \geq \kappa(Z, nH + K_Z) = 2.
\end{align*}
This completes the proof.

\subsection{Remarks on the proof}\label{can-bdl-formula}
Our strategy to prove $C_{n,m}$ heavily relies on the non-vanishing of $S^{0}(X_{\bar{\eta}}, lK_{X_{\bar{\eta}}})$. However, this often fails, for example when $X_{\bar{\eta}}$ is a supersingular elliptic curve. We have known that $S^{0}(X_{\bar{\eta}}, lK_{X_{\bar{\eta}}}) \neq 0$ if $K_{X_{\bar{\eta}}}$ is big and $l\gg0$.  To overcome this difficulty, an idea is to consider the relative Iitaka fibration $h: X \rightarrow Z$ as in the proof of Corollary  \ref{app-to-3dim-special}, then reduce to studying $\kappa(Z, K_Z + \Delta_Z)$ where $K_Z + \Delta_Z$ is relatively big over $Y$. To carry out this idea, we only need to have

``Canonical bundle formula'': for a fibration $h:(X, \Delta) \rightarrow Z$ from a klt pair such that $K_{X} + \Delta$ is relatively $\mathbb{Q}$-trivial over $Z$,
then there exists an effective divisor $\Delta_Z$ on $Z$ such that $K_{X} + \Delta \sim_{\mathbb{Q}} h^*(K_{Z}+  \Delta_Z)$.

Over complex numbers this is true, $(Z, \Delta_Z)$ can even be assumed to be klt up to some birational modifications, more precisely $\Delta = B + M$ is the sum of discriminant part and moduli part (cf. \cite[Theorem 4.5]{FM} and \cite[Theorem 0.2]{Amb}). In positive characteristic, we have such a canonical bundle  formula when the geometric generic fiber $(X_{\bar{\xi}}, \Delta_{\bar{\xi}})$ of $h$ is globally $F$-split (cf. \cite[Theorem 3.18]{Ej} or \cite[Theorem B]{DS}), or when $\Delta=0$ and $X_{\bar{\xi}}$ is a smooth elliptic curve. In general, the canonical bundle formula as above does not hold in positive characteristic, one can construct a counter example by a ruled surface over a curve with a multiple section purely inseparable over the base. Recently Witaszek \cite{Wit17} proves a weaker canonical bundle formula for fibrations of relative dimension one and gives some interesting applications. But his formulation does not seem to fit the above strategy.

\end{document}